\documentclass[11pt]{amsart}
\usepackage{extpfeil}
\usepackage[all]{xy}
\usepackage{amssymb} 
\usepackage[top=1.2in, bottom=1.2in, left=1.2in, right=1.2in]{geometry}
\usepackage{euscript}
\usepackage{tikz}
\usetikzlibrary{arrows,cd}
\tikzset{>=latex}
\usepackage{mathrsfs}
\usepackage{mathtools}
\usepackage{amsthm}
\usepackage{mathrsfs}
\usepackage{url}
\usepackage{color}

\numberwithin{equation}{section}

\theoremstyle{plain}
	\newtheorem{thm}[equation]{Theorem}
	\newtheorem{prop}[equation]{Proposition}
	\newtheorem{lem}[equation]{Lemma}
        
        \newtheorem{conj}[equation]{Conjecture}

	\newtheorem{cor}[equation]{Corollary}
	\newtheorem{lem/defn}[equation]{Lemma/Definition}

\theoremstyle{definition}
	\newtheorem{defn}[equation]{Definition}
	\newtheorem{ex}[equation]{Example}
	
	\newtheorem{properties}[equation]{Properties}
	\newtheorem{notation}[equation]{Notation}

\theoremstyle{remark}
	\newtheorem{rem}[equation]{Remark}

\def\ov#1{\overline{#1}}
\def\mf{\on{mf}}        
\def\l{\lambda}
\def\and{ \text{ and }}

\def\nc{\newcommand}
\def\on{\operatorname}

\def\p{{\tfrac{n}{2} + 1}}

\nc{\novertwo}{{\tfrac{n}{2}}}

\nc{\edit}[1]{\marginpar{\footnotesize{#1}}}

\nc{\bone}{{\mathbf 1}}

\nc{\C}{\mathbb{C}}
\nc{\Q}{\mathbb{Q}}
\nc{\Z}{\mathbb{Z}}
\nc{\PP}{\mathbb{P}}
\nc{\R}{\mathbb{R}}
\nc{\LL}{\mathbb{L}}
\nc{\OO}{\mathcal{O}}

\nc{\X}{\EuScript{X}}
\nc{\cC}{\EuScript{C}}
\nc{\cA}{\EuScript{A}}
\nc{\sZ}{\EuScript{Z}}

\nc{\id}{{\on{id}}}
\nc\Hom{{\on{Hom}}}
\nc\cone{{\on{cone}}}
\nc\Ob{{\on{Ob}}}
\nc\Spec{{\on{Spec}}}
\nc\Mod{{\on{Mod}}}
\nc\Perf{{\on{Perf}}}
\nc\End{{\on{End}}}
\nc{\into}{\hookrightarrow}
\nc{\tr}{\on{tr}}
\nc{\ev}{\on{ev}}
\nc{\im}{\on{im}}
\nc{\Mot}{\on{Mot}}
\nc{\pt}{\on{pt}}
\nc{\coker}{\on{coker}}
\nc{\rk}{\on{rank}}
\nc{\TOP}{\on{Top}_{\mathbb{C}}^{s}}
\nc{\gr}{\on{gr}}
\nc{\Catperf}{\text{Cat}^{\text{perf}}}
\nc{\Sym}{\on{Sym}}
\nc{\xra}{\xrightarrow}
\nc{\xora}{\xtwoheadrightarrow}

\nc{\lra}{\xleftarrow}
\nc{\Bet}{\mathbf{Betti}_{X}}
\nc{\codim}{\on{codim}}
\nc{\Fred}{\on{Fred}}
\nc{\colim}{\on{colim}}
\nc{\KK}{{\bf K}}
\nc{\Sp}{\on{Sp}}
\nc{\onto}{\twoheadrightarrow}
\nc{\A}{\mathbb{A}}
\nc{\Aff}{\on{Aff}}
\nc{\SH}{\on{SH}}
\nc{\QCoh}{\on{QCoh}}
\nc{\Alg}{\on{Alg}}
\nc{\Br}{\on{Br}}
\nc{\ta}{\widetilde{\a}}
\nc{\Shv}{\on{Shv}}
\nc{\GG}{\mathbb{G}}
\nc{\red}{\color{red}}
\nc{\an}{\on{an}}
\nc{\D}{\on{D}}
\nc{\qc}{\on{qc}}
\nc{\op}{\on{op}}
\nc{\shEnd}{{\mathcal End}}
\nc{\Sph}{\mathbb{S}}
\nc{\Top}{\on{Top}}
\nc{\alg}{\on{alg}}
\nc{\Map}{\on{Map}}
\nc{\Vect}{\on{Vect}}
\nc{\holim}{\on{holim}}
\nc{\GL}{\on{GL}}
\nc{\can}{\on{can}}
\nc{\deRham}{{\on{deRham}}}
\nc{\deR}{{\on{deR}}}
\nc{\polar}{\on{polar}}

\newcommand\blfootnote[1]{%
  \begingroup
  \renewcommand\thefootnote{}\footnote{#1}%
  \addtocounter{footnote}{-1}%
  \endgroup
}

\def\cF{\mathcal{F}}

\def\top{\on{top}}
\def\a{\alpha}

\def\th{\on{th}}
\def\Perf{\on{Perf}}

\def\Sp{\on{Sp}}

\def\G{\Gamma}
\def\b{\beta}

\def\QCoh{\on{QCoh}}

\def\bu{\bullet}

\def\an{\on{an}}
\def\nc{\on{nc}}

\def\ch{\on{ch}}

\def\o{\omega}
\def\e{\varepsilon}
\def\fm{\mathfrak{m}}

\def\Proj{\on{Proj}}

\def\del{\partial}
\def\ch{\on{ch}}
\def\dR{\on{dR}}
\def\Hdg{\on{Hdg}}
\def\HN{HN}
\def\HP{HP}
\def\gr{\on{gr}}
\def\pHS{\on{pHS}}
\def\Db{\on{D}^{\on{b}}}

\def\Dsingdg{\on{D}^{\on{sg}}}

\def\prim{{\on{prim}}}
\def\xla{\xleftarrow}

\def\DbZ{\D^{\on{b}, Z}}

\def\nc{\on{nc}}
\def\top{\on{top}}
\def\even{{\mathrm{even}}}
\def\odd{{\mathrm{odd}}}
\def\N{\mathbb{N}}
\def\co{\colon}
\def\ce{\coloneqq}
\def\bA{\mathbb{A}}
\def\e{\varepsilon}
\def\MR#1{}
\def\BM{\on{BM}}
\makeatletter
\@namedef{subjclassname@2020}{%
  \textup{2020} Mathematics Subject Classification}
\makeatother

\title[Hodge structure on the singularity category]{The Hodge structure on the singularity category of a complex hypersurface}
\author{Michael K. Brown}
\author{Mark E. Walker}
\date{}

\begin{document}
\vspace{18mm} \setcounter{page}{1} \thispagestyle{empty}
\begin{abstract}
Given a complex affine hypersurface with isolated singularity determined by a homogeneous polynomial, we identify the noncommutative Hodge structure on the periodic cyclic homology of its singularity category with the classical Hodge structure on the primitive cohomology of the associated projective hypersurface. As a consequence, we show that the Hodge conjecture for the projective hypersurface is equivalent to a dg-categorical analogue of the Hodge conjecture for the singularity category.
\end{abstract}
\maketitle

\blfootnote{Brown was partially supported by NSF grant DMS-2302373 and Walker by NSF grant DMS-2200732.
Walker was also supported by NSF grant DMS-1928930 and by the Alfred P. Sloan Foundation under grant G-2021-16778, while 
he was in residence at the Simons Laufer Mathematical Sciences Institute during the Spring 2024 semester.}


\section{Introduction}
Katzarkov-Kontsevich-Pantev conjecture in~\cite{KKP} that the periodic cyclic homology of any smooth and proper $\C$-linear differential $\Z$-graded
category $\cC$ may be equipped with a ``noncommutative (nc) Hodge structure",
generalizing the pure Hodge structure on the cohomology of a smooth and proper complex variety. More precisely, the proposed nc Hodge structure on the $0^{\th}$ periodic cyclic homology of $\cC$, denoted $HP_0(\cC)$, is given by analogues of the Hodge filtration and rational structure on the cohomology of a smooth and proper complex variety: the former is the filtration of $HP_0(\cC)$ arising from the negative cyclic homology of~$\cC$,
and the latter is the image of the rationalized
topological Chern character map $ K_0^{\top}(\cC)_\Q \to HP_0(\cC)$.
The statement of the classical Hodge Conjecture generalizes to
dg-categories equipped with an nc Hodge structure: we say such a dg-category $\cC$ satisfies the ``nc Hodge condition''
provided the image of the rationalized algebraic Chern character map $K_0^{\alg}(\cC)_\Q \to HP_0(\cC)$
coincides with the space of Hodge classes.  
We refer the reader to Section \ref{hodgedg} for more details (see also~\cite[Section 5.2]{perry}). It is known that the Hodge conjecture holds for a smooth projective complex variety $Y$ if and only if the Hodge condition holds for the dg-enhancement of its derived category $\Db(Y)$; see Example~\ref{ex:classichodge} or \cite{lin}.

The past two decades have seen a flurry of work focused on developing the Hodge theory of singularity categories of hypersurfaces
(i.e. matrix factorization categories), see e.g.~\cite{BD, BFK1, BFK2, BP, BRTV, BW, BW1,BW2, CT, dyckerhoff, efimov2,  HLP, KP, kim1,pippi, PV,  segal, shk2, shklyarov1}.
In this paper, we show that the singularity category of a complex hypersurface with isolated singularity determined by a
homogeneous polynomial may be equipped with an nc Hodge structure, and we describe it in terms of
invariants arising in classical Hodge theory. More specifically,
our main goal is to prove Theorem~\ref{introthm} below (see Theorem \ref{rephrase} for a more precise statement).
Before stating it, we fix some notation that will be used throughout the paper:

\begin{notation} \label{intronotation}
 Let $f \in \C[x_0, \dots, x_{n+1}]$ be a nonzero homogeneous polynomial, $R$ the associated affine hypersurface $\C[x_0, \dots, x_{n+1}]/(f)$ of dimension $n+1$, and $\fm \coloneqq (x_0, \dots, x_{n+1})$ its homogeneous maximal ideal. Set
  $X \coloneqq \Proj(R) \subseteq \PP^{n+1}$, a projective hypersurface of dimension $n$. We assume $n$ is even; otherwise, $X$ has no interesting Hodge theory.
  We assume also that $R$ has an isolated singularity, i.e. that $X$ is smooth.
  Let $\Dsingdg(R)$ 
  denote a dg-enhancement of the singularity category of $R$ (we specify in Section~\ref{backgroundsection} which dg-enhancement we use). 
  We write $H^n_{\prim}(X)$ for
  the $n^{\th}$ primitive cohomology of $X$ with rational coefficients.
  That is, for $n \ge 2$, 
  $H^n_{\prim}(X) = \ker(H^n(X; \Q) \xra{L} H^{n+2}(X; \Q))$, where  $L$ is the Lefschetz operator; and for $n = 0$,
  $H^0_{\prim}(X) \coloneqq \widetilde{H}^0(X; \Q)$,  the $0^{\th}$ reduced rational cohomology group of $X$.
  Since $X$ is a smooth hypersurface, $H^n_{\prim}(X)$ may be identified with $\coker(H^n(\PP^{n+1}; \Q) \to H^n(X; \Q))$; see Example~\ref{primitive} for more details. 
  
\end{notation}

\begin{thm}
\label{introthm}
There is an isomorphism 
$$
HP_0(\Dsingdg(R)) \cong H^{n}_{\on{prim}}(X; \C)
$$
that identifies the nomcommutative Hodge structure associated to $\Dsingdg(R)$  with the pure Hodge structure on $H^{n}_{\on{prim}}(X)$.
This isomorphism is compatible with the Chern character maps from topological $K$-theory, and hence the classical Hodge conjecture holds for $X$ if and only if
$\Dsingdg(R)$ satisfies the nc Hodge condition (Definition~\ref{hodgeconditiondef}).
\end{thm}

We note that it is a consequence of a result of Orlov \cite[Theorem 2.5]{orlovcoh} that the Hodge conjecture for $X$ is also equivalent to the nc Hodge condition for the \emph{graded} singularity category of $R$; see Remark~\ref{rem:orlov} for more details.

\medskip

The dg-category $\Dsingdg(R)$ is smooth over $\C$; this follows by combining \cite[Theorem 6.3]{lunts} and \cite[Proposition 3.10(c)]{keller11}. However, it is not proper as a
differential $\Z$-graded category. By results of Buchweitz and Eisenbud, $\Dsingdg(R)$ is quasi-equivalent to the category $\mf(f)$ of matrix factorizations of $f$, and so it may be equipped canonically with the structure of a proper
differential $\Z/2$-graded category (see Section~\ref{ss:mf}); but in this theorem, $HP_0(\Dsingdg(R))$ refers to the periodic cyclic
homology of $\Dsingdg(R)$ viewed as a $\Z$-graded category.  Nevertheless, as we prove below, the Hochschild invariants of $\Dsingdg(R)$ have the necessary features (see Properties
\ref{assumptions}) to make sense of an nc 
Hodge structure on $HP_0(\Dsingdg(R))$ and to formulate an nc Hodge condition.

The Hodge Conjecture is known to hold for a projective hypersurface $X$ in the following cases~\cite[\S 2]{othershioda}:
\begin{enumerate}
\item $\dim(X)$ odd, trivially.
\item $\dim(X) = 2$ (by the Lefschetz 1-1 Theorem).
\item $\dim(X) = 4$ and $\deg(X) \le 5$~\cite{zucker, murre, contemurre}.
\item $X$ a Fermat hypersurface, under certain arithmetic conditions on the dimension and degree of $X$~\cite{ran, shioda}.
\end{enumerate}
We therefore conclude that $\Dsingdg(R)$ satisfies the nc Hodge condition in all of the above cases.
In Example~\ref{ex:cubic}, we explicitly compute the Hodge classes for $\Dsingdg(R)$ when $X$ is the 2-dimensional Fermat hypersurface of degree 3.

\medskip
As mentioned above, since $R$ is a hypersurface, we may replace $\Dsingdg(R)$ by the quasi-equivalent dg-category $\mf(f)$ of matrix factorizations of $f$; see Section \ref{ss:mf} for the definition. Thus, our main theorem may be recast as an isomorphism 
\begin{equation}
\label{eqn:mfiso}
HP_0(\mf(f)) \cong H^{n}_{\on{prim}}(X; \C)
\end{equation}
that preserves Hodge structures and is compatible with Chern character maps.

\medskip
Let us give an overview of the paper. 
We collect in Section \ref{backgroundsection} the necessary background and terminology in
order to state our main result precisely; see Theorem \ref{rephrase}. This includes constructing an explicit map from $H^n_\prim(X)$ to $HP_0(\Dsingdg(R)) \cong HP_0(\mf(f))$; see~\eqref{alphadef}.
In Section~\ref{sec:tech}, we recall (and extend) several results necessary for the proof of Theorem~\ref{introthm}. More precisely, we describe the quasi-equivalence relating $\Dsingdg(R)$ and
$\mf(f)$, and we establish important details regarding ``de Rham models'' for the
Hochschild, negative cyclic, and periodic homology complex of $\mf(f)$ and related dg-categories.  We also recall from the authors' previous paper \cite{devissage} an explicit
description of the  boundary map in a certain d\'evissage long exact sequence; this map plays a crucial role in relating the nc Hodge filtration on $HP_0(\Dsingdg(R)) \cong
HP_0(\mf(f))$ with the classical Hodge filtration on $H_\prim^n(X)$.

In Section~\ref{sec:wang}, we establish the following analogue of the Wang exact sequence of
a fibration over a circle (see Theorem \ref{wangthm} below for a more precise
statement, and see Remark \ref{wang} for an explanation of how this result
relates to the Wang exact sequence): 

\begin{thm}
\label{introwang}
  Let $k$ be a field of characteristic $0$, $Q$ a smooth $k$-algebra, and $f \in Q$ a non-zero-divisor.
There is a distinguished triangle
$$
HN(\mf(f)) \to HN^{\Z/2}(\mf(f)) \to HN^{\Z/2}(\mf(f)) \to
$$
of $k$-vector spaces, where $HN(\mf(f))$ (resp. $HN^{\Z/2}(\mf(f))$) denotes the negative cyclic complex of $\mf(f)$ considered as a differential $\Z$-graded (resp. $\Z/2$-graded) category.
\end{thm}

The map $HN^{\Z/2}(\mf(f)) \to HN^{\Z/2}(\mf(f))$ in Theorem~\ref{introwang} is an analogue of the endomorphism $T - \id$ on the cohomology of the Milnor fiber, where $T$ is induced by monodromy (see Remark~\ref{wang}). Theorem~\ref{introwang} thus closely resembles a result of Blanc-Robalo-To\"en-Vezzosi~\cite[Main Theorem]{BRTV} concerning the $\ell$-adic realization of singularity categories.

Section \ref{mainsection} contains the proof of Theorem~\ref{introthm}. A summary of the content of Sections 2 - 4 that is necessary for the proof of Theorem~\ref{introthm} is provided in Theorem~\ref{newtheorem}; readers who are already familiar with noncommutative Hodge theory and singularity categories may wish to skip directly to Theorem~\ref{newtheorem} and refer back to Sections 2 - 4 as needed. The most technical aspect of the proof of our main result
is verifying that the isomorphism
$HP_0(\Dsingdg(R)) \cong H^{n}_{\on{prim}}(X)$
identifies the nc Hodge filtration on $HP_0(\Dsingdg(R))$
with the classical Hodge filtration on $H^{n}_{\on{prim}}(X)$. Our approach is to identify both with an intermediate object:
the ``polar filtration" on $H^n(U)$, where $U$ is the complement of $X$ in $\PP^{n+1}$.
In Section \ref{examples}, we discuss some examples in the setting of Fermat hypersurfaces, applying work of Shioda~\cite{shioda}. 

\subsection*{Acknowledgments} We thank the anonymous referee for many helpful comments that improved this paper.

\section{Background}
\label{backgroundsection}

\subsection{Hodge structures}
\label{classical}

The following definition is non-standard, but it will be useful in this paper.

\begin{defn}
  A \emph{pre-Hodge structure} is a pair $V = (V_\Q, F^\bu V_\C)$, where $V_\Q$ is a finite dimensional $\Q$-vector space,
  and $F^\bu V_\C$ is a filtration of $V_\C \coloneqq V_\Q \otimes_\Q \C$ that is decreasing, complete, and exhaustive:
  $F^p V_\C \subseteq F^{p-1} V_\C$ for all $p$, $F^p V_\C = 0$ for $p \gg 0$, and $F^p V_\C = V_\C$ for
  $p \ll 0$. 
  A \emph{morphism of pre-Hodge structures} $V \to V'$ is a $\Q$-linear map $V_\Q \to V'_\Q$
  whose complexification respects the filtrations.
  \end{defn}

\begin{rem} \label{rem71} An isomorphism of pre-Hodge structures $V$ and $V'$ is determined by an isomorphism $\a: V_\C \xra{\cong} V'_\C$ 
such that $\alpha(F^pV_\C) = F^pV'_\C$ for all $p$, and $\alpha(V) = V'$.

\end{rem}

The notion of a pre-Hodge structure is a weakening of the classical notion of a pure Hodge structure, whose definition we now recall:

\begin{defn}
  Let $n \in \Z$. A \emph{pure Hodge structure of weight $n$} is a pre-Hodge structure $V$
  with the property that, for all $p, q$ with $p + q = n+1$, we have $F^pV_\C \oplus \overline{F^qV_\C} = V_\C$, where the overline denotes
  complex conjugation. A \emph{morphism} of pure Hodge structures is a morphism of the underlying pre-Hodge structures. 
\end{defn}

Given a pre-Hodge structure $V$ and $m \in \Z$, we write $V(m)$ for its {\em $m^{\th}$ Tate twist}, which is
defined by setting $V(m)_\Q = V_\Q$, $V(m)_\C = V_\C$ and $F^p V(m)_\C = F^{p + m} V_\C$. 
If $V$ is pure of weight $n$, then $V(m)$  is pure of weight $n - 2m$.

\begin{ex}
\label{ex:classical}
Let $X$ be a smooth, proper complex variety, and let $V_\Q = H^{j}(X ; \Q)$, the singular cohomology of $X$ with rational coefficients.
Equip $V_\C = H^{j}(X ; \C)$ with the filtration given by
$$
F^pV_\C \coloneqq \im(H^{j}(X, \tau^{\ge p} \Omega^\bu_{X / \C}) \to H_{\dR}^{j} (X ; \C) \cong H^{j}(X ; \C)),
$$
where $\tau^{\ge p}\Omega^\bu_{X / \C}$ denotes the brutal truncation of the de Rham complex in cohomological degrees $\ge p$, and
$H_{\dR}^{j} (X ; \C)$ denotes the $j^{\th}$ hypercohomology of $\Omega^\bu_{X / \C}$. It is a classical result that $(V_\Q, F^\bu V_\C)$
is a pure Hodge structure of weight $j$. The $m^{\th}$ Tate twist of this Hodge structure is written as  $H^j(X ; \Q(m))$.
\end{ex}

\begin{ex}
  \label{primitive}
  Let $X$ be a smooth, projective complete intersection of codimension $c$ in $\PP^{n+c}$.
  That is, $X = \Proj(R)$,
  where $R= \C[x_0, \dots, x_{n+c}]/(f_1, \dots, f_c)$ for a
  regular sequence of homogeneous polynomials $f_1, \dots, f_c$  such that  $R$ has an isolated singularity. Assume also that $n$ is even. As with any smooth, projective variety, 
the  primitive cohomology of $X$ may be equipped with a pure Hodge structure;
let us recall the definition of primitive cohomology. In the case where $n = 0$ (so that $X$ is a collection of points),
we define $H_{\prim}^{0}(X) = \widetilde{H}^{0}(X ; \Q)$. Assume $n \ge 2$. We let
$$
L : H^*(X ; \Q) \to H^{*+ 2}(X ; \Q(1))
$$
denote the \emph{Lefschetz operator}, i.e. the map given by multiplication by the class in $H^2(X ; \Q(1))$ of a generic hyperplane section of $X$.
It is a morphsim of pure Hodge structures.
For $0 \le j \le n$, the $j^{\th}$ \emph{primitive cohomology} of $X$ is 
$$
H^j_{\prim}(X) \coloneqq \ker(L^{n+1-j} : H^j(X ; \Q) \to H^{2n +2 - j}(X ; \Q(n+1-j)).
$$
As $L$ is a morphism of pure Hodge structures, $H^j_{\prim}(X)$ acquires a pure Hodge structure, and it is pure of weight $j$. 

Since we assume that $X$ is a smooth complete intersection of even dimension, the Hard Lefschetz Theorem gives
$$
H^j(X ; \Q) = \begin{cases} 0, & j \text{ odd}; \\ L^{j/2} \cdot H^0(X ;\Q(-j/2)), & j \text{ even, $j \ne n$.}  \end{cases}
$$
In other words, the map $i^*: H^j(\PP^{n+c}; \Q) \to H^j(X; \Q)$ induced by the canonical embedding $i: X \into \PP^{n+c}$ is 
an isomorphism for all $j \ne n$; in particular, $H^j_{\prim}(X) = 0$ unless $j = n$. So, the only ``interesting" cohomology lies 
in degree $n$, and in that degree we have a canonical decomposition of Hodge structures
$$
H^n(X ; \Q) = H^n_{\prim}(X ; \Q) \oplus  L^{n/2} \cdot H^0(X ; \Q(-n/2)).
$$
Since the summand $L^{n/2} \cdot H^0(X ; \Q(-n/2))$ equals the image of the map
$$
L^{n/2} \cdot H^0(\PP^{n+c}; \Q(-n/2) = H^n(\PP^{n+c} ; \Q) \xra{i^*} H^n(X ; \Q)),
$$
we have a canonical isomorphism
$$
H^n_{\prim}(X ; \Q) \cong \coker(H^n(\PP^{n+c} ; \Q) \xra{i^*} H^n(X ; \Q))
$$
of pure Hodge structures.
\end{ex}

\subsection{Hodge structures associated to dg-categories}
\label{hodgedg}

Let $\cC$ be a $\C$-linear differential $\Z$-graded category (or dg-category, for short). As discussed in the introduction, it follows from work of Katzarkov-Kontsevich-Pantev in~\cite{KKP} that one may associate a pre-Hodge structure to $\cC$ whenever $\cC$ enjoys certain properties that resemble features of the bounded derived category of a smooth, proper complex variety. Let us now explain this in detail. 

We write $HH_*(\cC)$, $HN_*(\cC)$, and $HP_*(\cC)$ for the Hochschild, negative cyclic, and periodic cyclic homology of $\cC$;
we refer the reader to e.g.~\cite[Section 3]{BW} for the definitions of these invariants. We recall that $HN_*(\cC)$ is a $\C[u]$-module with $u$ an indeterminate of homological degree $-2$,
determined by the identification $HN_*(\C) = \C[u]$, and $HP_*(\cC) = HN_*(\cC) \otimes_{\C[u]} \C[u, u^{-1}]$. 
There is a notion of topological $K$-theory for dg-categories, developed by Blanc~\cite[Definition 4.13]{blanc}; let $K_*^{\top}(\cC)$
denote the topological $K$-theory groups of $\cC$, and set $K_*^{\top}(\cC)_\Q \coloneqq K^{\top}_*(\cC) \otimes_\Z \Q$. Topological $K$-theory and periodic cyclic homology are related via
Blanc's \emph{topological Chern character map} 
$
\ch^{\top} : K_*^{\top}(\cC)_\Q \to HP_*(\cC)
$
\cite[Section 4.4]{blanc}.

\begin{notation}
\label{notation}
Given a noetherian $\C$-scheme $Y$ with enough locally free sheaves, let $\Db(Y)$ and $\Perf(Y)$ denote dg-enhancements of the bounded derived category of $Y$ and category of perfect complexes on $Y$, respectively; all such dg-enhancements are unique up to a sequence of quasi-equivalences~\cite{LO10}. We write 
\begin{align*}
K_*^{\top}(Y) \coloneqq K_*^{\top}(\Perf(Y)), & \quad HP_*(Y) \coloneqq HP_*(\Perf(Y)), \\
HN_*(Y) \coloneqq HN_*(\Perf(Y)), & \quad HH_*(Y) \coloneqq HH_*(\Perf(Y)).
\end{align*}
We adopt the widely-used notation $K_*(Y)$ and $G_*(Y)$ for the
(non-connective) algebraic $K$-theory groups of $\Perf(Y)$ and
$\Db(Y)$; see \cite[Section 3.2.32]{schlichting11} for background on
algebraic $K$-theory of dg-categories. 
We also write
\begin{align*}
G_*^{\top}(Y) \coloneqq K^{\top}_*(\Db(Y)), & \quad HP^{\BM}_*(Y) \coloneqq HP_*(\Db(Y)), \\
HN^{\BM}_*(Y) \coloneqq HN_*(\Db(Y)), & \quad HH^{\BM}_*(Y) \coloneqq HH_*(\Db(Y));
\end{align*}
here, ``$\BM$" stands for ``Borel-Moore". Given a commutative ring $A$, we write $G_*(A) \coloneqq G_*(\Spec(A))$, and similarly for the other invariants discussed here.
\end{notation}

We recall the notions of smoothness and properness for dg-categories:
\begin{defn}
\label{sandp}
The dg-category $\cC$ is \emph{smooth} if $\cC$ is perfect as a $\cC$-$\cC$-bimodule, 
and it is \emph{proper} if $\dim_\C H^*\Hom_\cC(C, C') < \infty$ 
for all objects $C, C'$ of $\cC$.
\end{defn}
When $X$ is a separated scheme of finite type over $\C$, it follows from (the proof of)~\cite[Proposition 3.31]{orlov} that $X$ is smooth (resp. proper) if and only if the dg-category $\Perf(X)$ of perfect complexes of $\OO_X$-modules is smooth (resp. proper).

In this paper, the dg-categories we consider will not always be smooth and proper. We will be interested in dg-categories that satisfy the following conditions, which are exactly what one needs to equip it with a pre-Hodge structure.

\begin{properties}
\label{assumptions}
For a dg-category $\cC$, we consider the following properties.
\begin{enumerate}
\item $\dim_\C HP_0(\cC) < \infty$.
\item The filtration of $HP_0(\cC)$ given by
$$
F^p_{\nc}HP_0(\cC) = \im\left(HN_{2p}(\cC) \xra{\can} HP_{2p}(\cC)
\xra{u^p} HP_0(\cC)\right).
$$
satisfies $F^p_{\nc}HP_0(\cC) = 0$ for $p \gg 0$ and $F^p_{\nc}HP_0(\cC) = HP_0(\cC)$ for $p \ll 0$.
\item $\im(\ch^{\top}) \otimes_\Q \C = HP_0(\cC)$.
\end{enumerate}
\end{properties}

\begin{prop}
\label{1and2}
Properties (1) and (2) hold for any dg-category that is smooth and proper over $\C$.
\end{prop}

\begin{proof}
  By~\cite[Proposition 8.2.3]{KS}, we have $\dim_\C HH_*(\cC) < \infty$. We also have non-canonical isomorphisms $HN_*(\cC) \cong HH_*(\cC)[u]$ and $HP_*(\cC) \cong HH_*(\cC)[u, u^{-1}]$;
  these  follow from Kaledin's noncommutative Hodge-to-de-Rham degeneration theorem~\cite{kaledin}. Properties (1) and (2) follow immediately. 
\end{proof}

Property (3) is conjectured by Blanc  to hold for any smooth and proper dg-category: 

\begin{conj}[The Lattice Conjecture,~\cite{blanc} Conjecture 1.7]
\label{latticeconj}
If $\cC$ is smooth and proper over $\C$, then the map $K_*^{\top}(\cC)_\C \to HP_*(\cC)$ induced by $\ch^{\top}$ is an isomorphism. 
\end{conj}

In fact, the Lattice Conjecture is known to hold for many dg-categories that are not smooth or proper: we refer the reader to \cite[Section 5]{BSLattice} for a list of known cases of the Lattice Conjecture. 

\begin{defn}
  Assume the dg-category $\cC$ satisfies Properties \ref{assumptions}. The \emph{nc pre-Hodge structure} for $\cC$, written $\pHS(\cC)$,
  is the pair $(V_\Q, F^\bullet V_\C)$, where $V_\Q \ce \im(\ch^{\top}: K^{\top}_0(\cC)_\Q \to HP_0(\cC))$, and $F^p V_\C \ce F^p_{\nc} HP_0(\cC)$.
The filtration $F_{\nc}^\bullet HP_0(\cC)$ is called the \emph{noncommutative Hodge filtration}, or \emph{nc Hodge filtration} for short.
\end{defn}

\begin{ex}
\label{isoclassical}
Let $X$ be a smooth, proper complex variety, and take $\cC = \Db(X)$. Since $\Db(X)$ is smooth and proper, Properties (1) and (2) hold for $\cC$. Property (3) also holds for $\cC$, since the Lattice Conjecture holds in this case~\cite{blanc}. We have a commutative diagram
$$
\xymatrix{
HN_{2p}(X) \ar[r]^-{\cong} \ar[d]^-{u^p} &  \bigoplus_{j \in \Z} H^{2j}(X, \tau^{\ge j+p} \Omega^\bu_{X/\C}) \ar[d] \\
HN_0(X) \ar[r]^-{\cong} \ar[d] &  \bigoplus_{j \in \Z} H^{2j}(X, \tau^{\ge j} \Omega^\bu_{X/\C}) \ar[d] \\
HP_0(X) \ar[r]^-{\cong} &  \bigoplus_{j \in \Z} H^{2j}_{\dR}(X ; \C) \\
K_*^{\top}(X) \ar[u]^-{\ch^{\top}} \ar[r]^-{\cong} &  \ar[u]^-{\ch^{\top}} KU^*(X),
}
$$
where $KU^*(X)$ denotes the topological $K$-theory of $X$. The first three horizontal isomorphisms are given by combining theorems of Keller and Weibel (\cite[Section 5.2]{keller},~\cite[Theorem 3.3]{weibel}), and the bottom isomorphism is due to Blanc~\cite[Theorem 1.1(b)]{blanc}. A straightforward calculation shows that the top and middle squares commute, and the bottom square commutes by~\cite[Proposition 4.32]{blanc}. We conclude that there is a natural isomorphism of pre-Hodge structures
$$
\pHS(\Db(X)) \cong \bigoplus_{p \in \Z} H^{2j}(X, \Q(j)).
$$ 
In particular, $\pHS(\Db(X))$ is a pure Hodge structure of weight 0. See also \cite{tu24}, where a more detailed comparison of the classical Hodge structure on the cohomology of $X$ and the nc Hodge structure on $\Db(X)$ is carried out. 
\end{ex}

\subsection{The nc Hodge condition for a dg-category}
We begin by recalling the statement of the Hodge conjecture. Let $X$ be a smooth, projective complex variety and $K_*(X)_\Q \coloneqq K_*(X) \otimes \Q$ the rationalized algebraic $K$-theory groups of $X$. The classical Hodge conjecture proposes a description of the image of the Chern character map
\begin{equation}
\label{chern}
\ch \colon K_0(X)_\Q \to \bigoplus_{p \in \Z} H^{2p} (X ; \C).
\end{equation}
We set 
$$
  \Hdg^{2p}(X) \ce H^{2p}(X ; \Q) \cap F^p H^{2p}(X ; \C) = H^{2p}(X ; \Q(p)) \cap F^0 H^{2p}(X, \C(p)),
  $$
  and we write
$\Hdg(X) \coloneqq \bigoplus_{p \in \Z} \Hdg^{2p}(X)$.
Elements of $\Hdg(X)$
are called \emph{Hodge classes}. 
It is well-known that the Chern  character map \eqref{chern} takes values in $\Hdg(X)$;
the Hodge Conjecture predicts that the image of the Chern character map \eqref{chern} is precisely $\Hdg(X)$.

The statement of the Hodge conjecture can be extended to any dg-category $\cC$ enjoying Properties \ref{assumptions}. Let $K_0(\cC)$ denote the Grothendieck group of $\cC$ and
$
\ch_{HN} : K_0(\cC) \to \HN_0(\cC)
$
the associated Chern character map; see, for instance,~\cite[Section 4]{BW} for the definition of $\ch_{HN}$. Composing with the natural map $\HN_0(\cC) \to \HP_0(\cC)$, one also obtains
$$
\ch_{HP} : K_0(\cC) \to \HP_0(\cC).
$$
As above, let $K_0(\cC)_\Q \coloneqq K_0(\cC) \otimes_\Z \Q$. By~\cite[Theorem 1.1(d)]{blanc}, the maps $\ch_{HN}$ and $\ch^{\top}$ are related by a commutative square
\begin{equation}
\label{dia}
\xymatrix{
K_0(\cC)_\Q \ar[r]^-{\ch_{HN}} \ar[d] & \HN_0(\cC) \ar[d]^-{} \\
K^{\top}_0(\cC)_\Q \ar[r]^-{\ch^{\top}} & \HP_0(\cC),
}
\end{equation}
where the vertical maps are the canonical ones.

\begin{defn} \label{hodgeconditiondef}
  For a dg-category $\cC$ that satisfies Properties~\ref{assumptions}, the subspace $\Hdg(\cC) \subseteq \HP_0(\cC)$ of \emph{Hodge classes of $\cC$} is defined to be
  $$
  \Hdg(\cC) \coloneqq
  \im\left(\ch^{\top}: K^{\top}_0(\cC)_\Q \to HP_0(\cC)\right) \cap F_{\nc}^0 HP_0(\cC).
  $$
  In other words,  $\Hdg(\cC) = \Hdg(\pHS(\cC))$, where for any pre-Hodge structure
  $V$, we set $\Hdg(V) = V_\Q \cap F^0 V_\C$. By the commutativity of Diagram \eqref{dia}, the Chern character map $\ch_{HP}$, which is given by composing the top and right-most maps in \eqref{dia}, takes values in $\Hdg(\cC)$. We say $\cC$ satisfies the \emph{nc Hodge condition} if $\im(\ch_{HP}) = \Hdg(\cC)$.
\end{defn}

\begin{ex}
\label{ex:classichodge}
When $X$ is a smooth, projective complex variety, the isomorphism of pure Hodge structures in Example \ref{isoclassical} yields a natural isomorphism $\Hdg(\Db(X)) \cong \Hdg(X)$. Moreover, this isomorphism makes the triangle
$$
\xymatrix{
K_0(X)_\Q \ar[r]^-{\ch}  \ar[rd]_-{\ch_{HP}}& \Hdg(X) \ar[d]^-{\cong}\\
 & \Hdg(\Db(X)) \\
}
$$
commute. It follows that the Hodge conjecture holds for $X$ if and only if $\Db(X)$ satisfies the nc Hodge condition: this was also recently proven by Lin~ \cite{lin}. 
\end{ex}

\subsection{Statement of the main theorem}
\label{singsec}

To state our main result (Theorem~\ref{rephrase}), we need the following two technical results.

\begin{prop}
\label{point}
With $R$ and $\fm$ as in Notation~\ref{intronotation}, set $W \coloneqq \Spec(R)~\setminus~\{\fm\}$. The canonical maps
$$
K_0(W) \to G_0(W), \quad
K_0^{\top}(W) \to G_0^{\top}(W), \quad \text{and} \quad
HP_0(W) \to HP^{\BM}_0(W)
$$
are isomorphisms,
and so are the maps 
$$
G_0(R) \to G_0(W), \quad G_0^{\top}(R) \to G_0^{\top}(W) \quad \text{and} \quad
HP_0^{\BM}(R) \to HP^{\BM}_0(W) 
$$ 
induced by pullback along the natural map $W \into \Spec(R)$.
\end{prop}

\begin{proof} The first batch of isomorphisms hold since $W$ is regular, by assumption.
  Since $G$-theory satisfies d\'evissage, we have a right exact sequence
  $$
G_0(R/\fm) \to G_0(R) \to G_0(W) \to 0.
$$
By~\cite[Lemma 13.4]{yoshino}, the pushforward map $G_0(R/\fm) \to G_0(R)$ is 0; this proves the result for $G$-theory.

We now address $HP^{\BM}$; the proof involving $G^{\top}$ is nearly identical. 
By~\cite[Theorem A.2]{khan} (see also \cite[Example 4.8]{devissage}),
there is a d\'evissage quasi-isomorphism $HP^{\BM}(R/\fm) \to HP(\D^{\on{b}, \{\fm\}}(R))$, where $\D^{\on{b}, \{\fm\}}(R)$ denotes the subcategory of $\D^{\on{b}}(R)$ given by objects with support contained in $\{\fm\}$.
(D\'evissage also holds for topological $K$-theory, as observed in~\cite[Example 2.3]{HLP}.)
There is thus a localization exact triangle
$$
HP^{\BM}(R/\fm) \to HP^{\BM}(R) \to HP^{\BM}(W) \to;
$$
see e.g.~\cite[Lemma 2.8]{devissage}. Since $HP_{-1}^{\BM}(R/\fm) \cong  HP_{-1}(\C) = 0$, it suffices to show that the pushforward map
$HP_0(R/\fm) \to HP^{\BM}_0(R)$ is 0. To prove this, we use that the Chern character map is natural for dg-functors,
so that we have a commutative square
$$
\xymatrix{
G_0(R/\fm) \ar[r] \ar[d]^-{\ch_{HP}} & G_0(R) \ar[d]^-{\ch_{HP}} \\
HP^{\BM}_0(R/\fm) \ar[r] & HP^{\BM}_0(R).
}
$$
The left-hand map is isomorphic to $\ch_{HP} : K_0(\C) \to HP_0(\C)$, which induces an isomorphism $K_0(\C) \otimes_\Z \C \cong HP_0(\C)$. 
Once again applying~\cite[Lemma 13.4]{yoshino}, the top arrow in this square is the zero map, and so the bottom arrow must be zero as well.
\end{proof}

Let $R$ be as in Notation~\ref{intronotation}. The \emph{singularity category} of the ring $R$  is the dg-quotient  $\Dsingdg(R) \ce \Db(R) / \Perf(R)$. We note that the (triangulated) homotopy category of $\Dsingdg(R)$ need not have a unique dg-enhancement; see e.g. \cite[Example 8.24]{antieau18}.

\begin{prop} \label{HPsing} 
There are short exact sequence
$$
0 \to HP_0(\C) \to HP_0^{\BM}(R) \to HP_0(\Dsingdg(R)) \to 0
$$
and
$$
0 \to K^{\top}_0(\C) \to G^{\top}_0(R) \to K^{\top}_0(\Dsingdg(R)) \to 0,
$$
where the maps are induced by the extension of scalars functor $\Db(\C) \to \Db(R)$ and the canonical functor $\Db(R) \to \D^{\on{sg}}(R)$.
\end{prop}

\begin{proof} Let $E$ denote either $HP$ or $K^{\top}$. In both cases, $E$ is a localizing $\A^1$-homotopy invariant  such that $E_{-1}(\C) = 0$. 
  Since $R$ is $\Z_{\ge 0}$-graded, extension of scalars along $\C \to R$ induces
  an isomorphism $E_*(\C) \xra{\cong} E_*(R)$. 
Using that $E_{-1}(\C) = 0$,  the exact triangle 
$$
E(R) \to E(\Db(R)) \to E(\Dsingdg(R)) \to
$$
induces an exact sequence $E_0(\C) \to E_0(\Db(R)) \to E_0(\Dsingdg(R)) \to 0$.
Finally, we observe that $E_0(\C) \to E_0(\Db(R))$ is split injective. To see this, choose a smooth closed point $x = V(\fm) \in \Spec(R)$; extension of scalars along the map $R \to R / \fm \cong \C$ determines a functor $\Db(R) \to \Db(\C)$, yielding the desired splitting. 
\end{proof}

We have a composition
\begin{equation}
\label{Kcomposition}
K_0(X) \to K_0(W) \xra{\cong} G_0(R) \to K_0(\Dsingdg(R)),
\end{equation}
where the first map is induced by the fibration $W \to X$ with fiber $\C^*$, the second is the (inverse of the) isomorphism from Proposition \ref{point}, and the last is induced by the canonical map $\Db(R) \to \Dsingdg(R)$. The composition \eqref{Kcomposition} admits the following simpler description: given a vector bundle $\cF$ on $X$, write $\cF = \widetilde{M}$ for some graded $R$-module $M$. Since $M$ and $\cF$ pull back to the same sheaf on $W$, the composition \eqref{Kcomposition} sends the class $[\cF]$ to $[M] \in K_0(\Dsingdg(R))$.

We have compositions
\begin{equation}
\label{Ktopcomposition}
KU^0(X) \cong K_0^{\top}(X) \to K_0^{\top}(W)  \xra{\cong} G_0^{\top}(R)  \to K_0^{\top}(\Dsingdg(R))
\end{equation}
and
\begin{equation}
\label{HPcomp}
H^{\on{even}}(X; \C) \cong HP_0(X) \to HP_0(W) \xra{\cong} HP^{\BM}_0(R) \to HP_0(\Dsingdg(R))
\end{equation}
that are defined in the same way, except the first isomorphism in \eqref{Ktopcomposition} is given by Blanc's comparison isomorphism~\cite[Theorem 1.1(b)]{blanc}, and the first map in \eqref{HPcomp} is induced by the HKR isomorphism \cite[Theorem 3.4.4]{loday}. By Proposition \ref{HPsing}, \eqref{HPcomp} induces a map on reduced cohomology; restricting to primitive cohomology, we arrive at the map
\begin{equation} \label{alphadef}
\a : H^{n}_{\prim}(X; \C) \to HP_0(\Dsingdg(R)).
\end{equation} 
We may now precisely formulate our main result (Theorem \ref{introthm}) as follows:

\begin{thm} 
\label{rephrase}
Let $R = \C[x_0, \dots, x_{n+1}] / (f)$, where $f$ is a homogeneous polynomial such that $X = \Proj(R)$ is smooth. Assume $n$ is even.
\begin{enumerate}
\item The dg-category $\Dsingdg(R)$ enjoys Properties \ref{assumptions}. In particular, we may associate a pre-Hodge structure to $\Dsingdg(R)$.
\item The diagram
\begin{equation} \label{thediagram}
\begin{tikzcd}
  K_0(X) \ar[d, "\eqref{Kcomposition}"]\ar[r, "\mathrm{can}"] & KU^0(X)\ar[d, twoheadrightarrow, "\eqref{Ktopcomposition}"] \ar[r,"\ch^{\top}"] & H^{\on{even}}(X;\C) \ar[d, twoheadrightarrow, "\eqref{HPcomp}"] \\  
  K_0(\Dsingdg(R)) \ar[r, "\mathrm{can}"]  &  K^{\top}_0(\Dsingdg(R)) \ar[r,"\ch^{\top}"]& HP_0(\Dsingdg(R))\\
\end{tikzcd}
\end{equation}
commutes (where the maps denoted $\on{can}$ are the canonical maps), the middle and rightmost vertical maps are surjective as indicated,
and the images of $K_0(X)$ and $K_0(\Dsingdg(R))$ 
in $HP_0((\Dsingdg(R))$ 
coincide.

\item The map $\a$ defined in \eqref{alphadef} is an isomorphism of complex vector spaces that  induces an isomorphism 
$$
 H^{n}_{\prim}(X, \Q(\novertwo)) \xra{\cong} \pHS(\Dsingdg(R)) 
$$
of pre-Hodge structures (see Remark \ref{rem71}). In particular, the pre-Hodge structure $\pHS(\Dsingdg(R))$ is pure of weight $0$. 

\end{enumerate}
\end{thm}

As an immediate consequence of Theorem \ref{rephrase}, we have:

\begin{cor}
\label{hodgecondition}
The dg-category $\Dsingdg(R)$ satisfies the nc Hodge condition if and only if the Hodge conjecture holds for $X$.
\end{cor} 

\begin{proof} Theorem~\ref{rephrase}  gives the commutative square
$$
\begin{tikzcd}
  K_0(X)_\Q  \ar[r] \ar[d] &  \Hdg(H^n_\prim(X, \Q(\novertwo)) \ar[d, "\cong"] \\
  K_0(\Dsingdg(R))_\Q \ar[r] &  \Hdg(HP_0(\Dsingdg(R)), 
  \end{tikzcd}
$$
where the right vertical map is an isomorphism. The Hodge conjecture for $X$ (resp. nc Hodge condition for $\Dsingdg(R)$) is the assertion 
that the top (resp. bottom) horizontal map in this square is onto. Clearly, the Hodge conjecture for $X$ implies the nc Hodge condition for 
$\Dsingdg(R)$. The converse holds since Theorem~\ref{rephrase} also gives that the images of 
$K_0(X)_\Q$ and
$K_0(\Dsingdg(R))_\Q$ in $\Hdg(HP_0(\Dsingdg(R)) \subseteq HP_0(\Dsingdg(R)$ coincide.
\end{proof}

\begin{rem}
\label{rem:orlov}
The Hodge conjecture for $X$ is also equivalent to the nc Hodge condition for the \emph{graded} singularity category of $R$, i.e. the dg-quotient $\on{D}^{\on{sg}}_{\on{gr}}(R)$ of the bounded derived category of $\Z$-graded $R$-modules by its subcategory of perfect complexes. Indeed, this is nearly immediate from a result of Orlov \cite[Theorem 2.5]{orlovcoh}. Orlov's Theorem also implies that, when $X$ is Calabi-Yau, there is 
an equivalence of categories $\Db(X) \simeq \on{D}^{\on{sg}}_{\on{gr}}(R)$ and hence
an isomorphism $H^{\on{even}}(X ; \C) \cong HP_0(\on{D}^{\on{sg}}_{\on{gr}}(R))$ that preserves Hodge structures. However, while the categories $\Dsingdg(R)$ and $\on{D}^{\on{sg}}_{\on{gr}}(R)$ are closely related, 
we do not see a way to deduce  Theorem~\ref{rephrase} from these results concerning $\on{D}^{\on{sg}}_{\on{gr}}(R)$. In a bit more detail: there is a canonical functor $\on{D}^{\on{sg}}_{\on{gr}}(R) \to \Dsingdg(R)$ given by forgetting the grading, and in fact, by~\cite[Theorem 1.5]{tabuada} and \cite[Proposition A.8]{KMV}, there is a distinguished triangle
\begin{equation}
\label{tabuada}
E(\on{D}^{\on{sg}}_{\on{gr}}(R)) \to E(\on{D}^{\on{sg}}_{\on{gr}}(R)) \to E(\Dsingdg(R)) \to 
\end{equation}
for any localizing, $\A^1$-homotopy invariant $E$
of dg-categories taking values in a triangulated category. The middle map in \eqref{tabuada} is the canonical functor, and the first map is induced by the endofunctor $T - \id$ of $\on{D}^{\on{sg}}_{\on{gr}}(R)$, where $T$
denotes the grading twist by 1. However, the Hodge structure on a dg-category involves negative cyclic homology, which is not an $\A^1$-homotopy invariant; the triangle \eqref{tabuada} is therefore ultimately not useful for studying the Hodge structure of~$\Dsingdg(R)$.
\end{rem}

\section{Some intermediate results}
\label{sec:tech}

Before we embark on the proof of Theorem~\ref{rephrase}, we need some intermediate results of a technical nature.  Throughout this section, we let $k$ be a field of characteristic 0, $Q$ a smooth $k$-algebra, 
and $R$ a hypersurface ring of the form $Q/(f)$ for a non-zero-divisor $f \in Q$. We recall in this section the interpretation of  $\Dsingdg(R)$ as a dg-category of matrix factorizations of $f$, and also the ``de Rham models'' for $HH$, $HN$
and $HP$ of the latter. We also recall an explicit description of a certain boundary map occurring in a long exact d\'evissage sequence for $HP$.

\subsection{Matrix factorizations} \label{ss:mf}
A \emph{matrix factorization of $f$} is a finitely generated, $\Z/2$-graded projective  $Q$-module $F = F_{\ov{0}} \oplus F_{\ov{1}}$ equipped with an odd degree endomorphism $\del$ such that $\del^2$ coincides with
multiplication by $f$. Matrix factorizations were introduced by Eisenbud~\cite{eisenbud} in his study of maximal Cohen-Macaulay modules over hypersurface rings. Matrix factorizations of $f$ form the objects of a differential $\Z/2$-graded category $\mf(f)$, whose morphisms are the $\Z/2$-graded complexes
$
\Hom_{\mf}((F, \del), (F', \del')) \coloneqq \Hom_Q(F, F')$ with differential sending a homogeneous map $\alpha$ of degree $\overline{i} \in \Z/2$ to $\del' \alpha - (-1)^{i} \alpha \del$.
If $F$ is free, then $F_{\ov{0}}$ and $F_{\ov{1}}$ necessarily have the same rank, and so we may view $\del$ as a block matrix $\begin{pmatrix} 0 & A \\ B & 0\end{pmatrix}$, where $A$ and $B$ are square matrices such that $AB = BA = f \cdot \id$. In this case, we will sometimes denote matrix factorizations as pairs $(A, B)$.

By ``unfolding'' the $\Z/2$-grading, it is also possible to interpret $\mf(f)$ as a classical differential $\Z$-graded category, and we will use both points of view in this paper.
To keep this straight, it is useful to introduce a formal indeterminate $t$ of homological degree $-2$, and to identify $\Z/2$-graded vector spaces with $\Z$-graded modules over $k[t,
t^{-1}]$: given a $\Z/2$-graded $k$-vector space $V = V_{\ov{0}} \oplus V_{\ov{1}}$, the associated graded $k[t,t^{-1}]$-module is $\cdots \oplus V_{\ov{0}}t^{-1} \oplus V_{\ov{1}}t^{-1} \oplus
V_{\ov{0}} \oplus V_{\ov{1}} \oplus V_{\ov{0}} t \oplus V_{\ov{1}} t \oplus  \cdots$, and the inverse procedure is given by setting $t = 1$ and taking degrees modulo $2$. 
Using this identification, the {\em unfolding} of a $\Z/2$-graded vector space is restriction of scalars along $k \to k[t, t^{-1}]$. (There is also a ``folding'' procedure, given by extension of
scalars along this map, but it will not arise in this paper.) 

From this point of view, an object of $\mf(f)$ becomes a finitely generated $\Z$-graded projective module over the graded ring $Q[t, t^{-1}]$,
equipped with a $Q[t, t^{-1}]$-linear differential of degree $-1$  whose square
equals $ft$.  When we think of $\mf(f)$ as a $\Z/2$-graded dg-category, we are regarding it as a $k[t, t^{-1}]$-linear
dg-category, and its unfolding amounts to regarding it as merely a $k$-linear dg-category.

A key result used throughout this paper is that there is a quasi-equivalence of $Q$-linear ($\Z$-graded) dg-categories 
\begin{equation}
\label{BE}
\mf(f) \xra{\simeq} \Dsingdg(R).
\end{equation}
This follows from a combination of Theorems of Buchweitz and Eisenbud (\cite[Theorem 4.4.1]{buchweitz},~\cite[Corollary 6.3]{eisenbud}). Unlike $\mf(f)$, the dg-category
$\Dsingdg(R)$ cannot directly be realized as the unfolding of a $\Z/2$-graded dg-category. 

\begin{notation}
\label{nota:superscript}
Let $\cC$ be a $k[t, t^{-1}]$-linear dg-category, where $t$ has degree $-2$ (i.e. a $k$-linear differential $\Z/2$-graded category). We write $HN(\cC)$ (resp. $HN^{\Z/2}(\cC))$ for its Hochschild homology relative to $k$ (resp. $k[t, t^{-1}]$), and similarly for $HP$ and $HH$. 
\end{notation}

\subsection{de Rham models for Hochschild, negative cyclic, and periodic cyclic homology}

We will make use of explicit de Rham-type models for Hochschild, negative cyclic, and periodic cyclic homology of  $\mf(f)$ and related categories.
We begin with a technical point: 

\subsubsection{Adjoining a power series variable to a graded ring}
For a (homologically) $\Z$-graded vector space $W$, let $u$ be an indeterminate of degree $-2$, so that $W[u]$, and hence also $W[u]/u^m$, is $\Z$-graded. We set
$$
W [[ u ]] \ce \lim_m W[u]/u^m,
$$
where, importantly, the inverse limit is taken in the category of $\Z$-graded vector spaces. Thus, $W [[ u ]]$ is graded, and
for each $d$, its degree $d$ part is the subspace
$$
W [[ u ]]_d = \left\{\sum_{i \geq 0} v_i u^i \text{ : } v_i \in W_{2i+d}\right\}
$$
of the collection of  all power series with $W$ coefficients. 
Note that if $W$ is concentrated in degree $0$ (or, more generally, if $W_m = 0$ for $m \gg 0$), then $W [[ u ]]$ is really just a polynomial ring with $W$ coefficients. 
For instance, if we regard $k$ itself as being $\Z$-graded but concentrated in degree $0$, then the above definition of $k[[u]]$ yields $k[u]$. 
We will stick with the traditional notation $k[u]$ in this case.

The following is easily verified:

\begin{lem}
\label{lem:technical} Suppose $W$ is a $\Z$-graded $k$-vector space with only a finite number of nonzero degree components, and let $t$ be a degree $-2$ indeterminate.
  We have an identity 
  $$
  W[t, t^{-1}] [[ u ]] = W [[ v ]][t, t^{-1}],
  $$
  where $v \coloneqq ut^{-1}$, and $W [[ v ]]$ denotes all power series with $W$ coefficients in the degree $0$ variable $v$. 
\end{lem}

\begin{ex} \label{ex:technical} In particular, we have an identity $k[t,t^{-1}] [[ u ]] = k [[ v ]][t,t^{-1}]$, with $v \coloneqq ut^{-1}$ and $k[[ v ]]$
  the usual ring of power series. In other words, $k[t, t^{-1}][[u]]$ is a $\Z/2$-graded ring that is 0 in odd degree and a power series ring in even degree.
  We will use this identity frequently in this paper.   
\end{ex}

\begin{rem} The ring $k[t][[u]]$ may be identified with the $\Z$-graded polynomial ring $k[u, t]$, concentrated in negative, even degrees.
\end{rem}

\subsubsection{de Rham models}
\label{models}

Let $A$ be a smooth  $k$-algebra equipped with a $\Z$-grading (written homologically) such that $A_i = 0$ for all odd $i$,
and suppose $w \in A_{-2}$ is an element of degree~$-2$. We call such a pair $(A,w)$ a {\em smooth curved algebra}, and we define its {\em de Rham HN complex} to be 
$$
HN^{\deR}(A,w) \ce (\Omega^\bu_{A/k}[[u]], u d + \lambda_{dw}).
$$
Here, $u$ is an indeterminate of degree $-2$, and  $\Omega^\bu_A = \bigoplus_p \Omega^p_{A/k}$ is homologically graded
by declaring $|a_0 da_1 \cdots da_p| \ce p + \sum_i |a_i|$, where $|-|$ refers to the degree of a homogenous element.
In the formula for the differential, $u d + \lambda_{dw}$, the $d$ refers to the de Rham differential $d: \Omega^p_A \to \Omega^{p+1}_{A}$ (observe that it has
homological degree $+1$, so that $ud$ has degree $-1$) and $\lambda_{dw}$ refers to left multiplication by the degree $-1$ element $dw \in \Omega^1_A$.

We define the {\em de Rham HP complex} for the pair $(A,w)$ to be
$$
HP^{\deR}(A,w) \ce HN^{\deR}(A,w)[u^{-1}]  = (\Omega^\bu_{A/k}((u)), u d + \lambda_{dw}),
$$
where, in general, $W((u))$ is shorthand for $W[[u]][u^{-1}]$. The {\em de Rham HH complex} of $(A,w)$~is
$$
HH^{\deR}(A,w) \ce \frac{HN^{\deR}(A,w)}{u \cdot HN^{\deR}(A,w)} = (\Omega^\bu_A, \lambda_{dw}).
$$

If $A$ is a smooth $k[t, t^{-1}]$-algebra for a degree $-2$ indeterminate $t$ (i.e, a $\Z/2$-graded algebra),  we set
$$
HN^{\deR, \Z/2}(A, w) \ce (\Omega^\bu_{A/k[t, t^{-1}]}[[u]], u d + \lambda_{dw}).
$$
We also define $HP^{\deR, \Z/2}(A, w)$ and $HH^{\deR, \Z/2}(A, w)$ just as above. Observe that the complex $HN^{\deR, \Z/2}(A, w)$ is a dg-module over $k[t,t^{-1}][[u]] = k[[v]][t, t^{-1}]$ (see Example \ref{ex:technical}); that is, it is a
  differential $\Z/2$-graded module over the power series ring $k[[v]]$. When $A$ is a $\Z/2$-graded algebra, we can, and sometimes will, ignore $k[t, t^{-1}]$-linearity and consider the invariants $HN^{\deR}(A, w)$, $HP^{\deR}(A, w)$, and $HH^{\deR}(A, w)$ defined above.

If $A = A_0$ and $w = 0$, we write $HN^{\deR}(A, w)$ as just $HN^{\deR}(A)$, and if $Y = \Spec(A)$, we also write $HN^{\deR}(Y) \ce HN^{\deR}(A)$; we use the analogous notation for $HH$ and $HP$ as well. In this case, the classical Hochschild-Kostant-Rosenberg (HKR) isomorphism \cite[Theorem 3.4.4]{loday} induces quasi-isomorphisms
$$
HN(A) \xra{\simeq} HN^\deR(A), \quad HP(A) \xra{\simeq} HP^\deR(A), \quad \text{and} \quad HH(A) \xra{\simeq} HH^\deR(A),
$$
thus justifying the notation. More generally, for a smooth (ungraded) $k$-algebra  $Q$ and non-zero-divisor $f \in Q$, we may form the smooth curved algebras $(Q[t], ft)$ and $(Q[t, t^{-1}], ft)$.
For these, we have  the following HKR-type isomorphisms, which build on results in \cite{CT, PP, segal}:
\begin{thm} \label{thm:hkr} For a field $k$ of characteristic $0$, a smooth $k$-algebra $Q$, and a non-zero-divisor $f \in Q$,
  we have the following isomorphisms in the derived category of  dg-$Q[u]$-modules (in parts (2) and (3), we use Notation~\ref{nota:superscript}): 
\begin{enumerate}
\item $HN^{\BM}(Q/f)  \cong HN^{\deR}(Q[t], ft)$,
\item $HN(\Dsingdg(Q/f)) \cong HN(\mf(f)) \cong HN^{\deR}(Q[t,t^{-1}], ft)$, and 
\item $HN^{\Z/2}(\mf(f)) \cong HN^{\deR, \Z/2}(Q[t,t^{-1}], ft)$.
\end{enumerate}
\end{thm}

\begin{rem} The third isomorphism is a map of graded $Q[t,t^{-1}][[u]] = Q[[v]][t, t^{-1}]$-modules, or, equivalently, of $\Z/2$-graded $Q[[v]]$-modules.
\end{rem}

\begin{rem}  \label{rem:hkr}
  The isomorphisms in this theorem imply the analogous results involving both Hochschild and periodic cyclic homology upon modding out by and inverting $u$, respectively. 
\end{rem}

\begin{proof}[Proof of Theorem~\ref{thm:hkr}]
Part (1) follows from \cite[Proposition 2.16, Theorem 2.17]{devissage}. Parts (2) and (3) follow from \cite[Proposition 3.25, Theorem 3.31]{BW} and \cite{walker}, respectively.
\end{proof}

\subsection{An explicit calculation of a boundary map}
\label{sec:boundary}
Suppose $Z \into X$ is a closed immersion  of schemes of finite type over $\C$.
Let $\DbZ(X)$ denote the full dg-subcategory of $\Db(X)$ consisting of objects whose supports are contained in $Z$. 
We set $HP^{\BM, Z}(X) \ce HP(\DbZ(X))$. As already noted in the proof of Proposition~\ref{point}, a result of Khan~\cite[Theorem A.2]{khan} (see also \cite[Theorem 1.2]{devissage}) implies that the induced map
$$
HP^{\BM}(Z) \to HP^{\BM, Z}(X) 
$$
is a quasi-isomorphism: that is, $HP^{\BM}$ satisfies the d\'evissage property. By a result of Keller~\cite{kellercyclic}, we have a distinguished triangle  
$$
HP^{\BM, Z}(X) \to HP^{\BM}(X) \to HP^{\BM}(X \setminus Z) \to
$$
of dg-$k[u, u^{-1}]$-modules; see \cite[Section 2]{devissage} for details. Combining this with the d\'evissage property, we obtain 
a distinguished triangle 
\begin{equation}
\label{eqn:HPdevtri}
HP^{\BM}(Z) \to HP^{\BM}(X) \to HP^{\BM}(X \setminus Z) \to.
\end{equation}

\begin{rem}
\label{rem:Ktopdev}
By a result of Blanc, topological $K$-theory of dg-categories is a localizing invariant (see e.g. \cite[Theorem 2.6]{devissage}). Moreover, topological $K$-theory satisfies d\'evissage \cite[Example 2.3]{HLP}. Thus, we have a distinguished triangle 
\begin{equation}
\label{eqn:Kdevtri}
G^{\top}(Z) \to G^{\top}(X) \to G^{\top}(X \setminus Z) \to
\end{equation}
of spectra analogous to \eqref{eqn:HPdevtri}. We will make use of this in the proof of Theorem~\ref{rephrase}. 
\end{rem}

We apply the distinguished triangle \eqref{eqn:HPdevtri} in the following special case: for a smooth $k$-algebra $Q$ and a non-zero-divisor $f \in Q$, we obtain
a long exact sequence
\begin{equation} 
\label{E29}
\cdots \to HP_1(Q) \to HP_1(Q[1/f]) \xra{\del} HP^{\BM}_0(Q/f) \to HP_0(Q) \to \cdots.
\end{equation}
Using \cite[Theorem 5.5]{devissage}, one may explicitly
describe the boundary map $\del$ in terms of the de Rham models for periodic cyclic homology discussed in \ref{models}.
Before recalling this result, we note that every element of  $HP^\deR_1(Q[1/f])$ is represented by a finite sum of classes of the form $\frac{\a}{f^s} u^{\tfrac{p-1}2}$ for some integer $s \geq 1$
and some class $\a \in \Omega^p_Q$, with $p$ odd, such that $f d\a = s
df \a$ \cite[Lemma 5.3]{devissage}.

  \begin{thm}[\cite{devissage} Theorem 5.5] \label{deRhamModels}
    The boundary map
    $$
    \del \co HP_1(Q[1/f])\to HP^{\BM}_0(Q/f)
    $$
    in~\eqref{E29} corresponds, via the isomorphisms relating its source and target to their de Rham models,
    to the map
    $$
    \del^\deR: HP_1^{\deR}(Q[1/f]) \to HP^{\deR}_0(Q[t], ft)
    $$
    given by
$$
\del^\deR\left(\frac{\a}{f^s} u^{\tfrac{p-1}2}\right) = 
\frac{(-1)^{s}}{s!} d(\a t^s) u^{\frac{p+1}{2} - s}.
$$
\end{thm}

\subsection{Some calculations of Hochschild invariants of matrix factorization categories} \label{sec:calculations}

Formulas for the Hochschild, negative cyclic, and periodic cyclic homology of $\mf(f)$ relative to $k[t,t^{-1}]$ (as opposed to $k$) are well-known, due to work of Dyckerhoff~\cite{dyckerhoff}: 

\begin{thm}
\label{thm:dyckerhoff}
Let $k$ be a field of characteristic $0$, $Q$ a smooth $k$-algebra, and $f \in Q$ a non-zero divisor. Assume that the morphism $f: \Spec(Q) \to \A^1_k$ determined by $f$
has only one singular point, $\fm \in \Spec(Q)$, and it lies over the origin (i.e., $f \in \fm$). Set $d = \dim(Q_\fm)$, and
let $\Omega_f$ be the finite dimensional vector space $\frac{\Omega^d_{Q}}{df \cdot \Omega_{Q}^{d-1}}$. We will use Notation~\ref{nota:superscript}.

\begin{enumerate}
\item There is an isomorphism of $\Z/2$-graded $k$-vector spaces
$$
\Sigma^d \Omega_f \xra{\cong} HH^{\Z/2}_*(\mf(f)),
$$
where $d$ is considered mod 2. Moreover, under the identification of $HH^{\Z/2}_*(\mf(f))$ with $HH^{\deR, \Z/2}_*(\mf(f))$ given by Theorem \ref{thm:hkr}, the isomorphism is
induced by the canonical inclusion $\Omega^d_Q \subseteq \Omega^\bu_Q$. 

\item $HN^{\Z/2}_*(\mf(f))$ is free of finite rank as a $\Z/2$-graded $k[[v]]$-module, 
  $HP^{\Z/2}_*(\mf(f))$ is a finite dimensional $\Z/2$-graded vector space over $k((v))$, and both are
  concentrated in degree $d \pmod{2}$. (Recall that  $v = t^{-1}u$, and $k((v)) \ce k[[v]][v^{-1}]$.) 
\end{enumerate}
\end{thm}

\begin{rem} 
  Rephrasing in terms of $k[t, t^{-1}]$-modules, Theorem~\ref{thm:dyckerhoff}(1) means that we have an isomorphism
  $$
  \Sigma^d \Omega_f[t, t^{-1}] \cong HH^{\Z/2}_*(\mf(f))
  $$
  of graded $k[t,t^{-1}]$-modules induced by the inclusion $\Omega^d_Q[t, t^{-1}] \into \Omega^{\bullet}_Q[t, t^{-1}]$,
  and Theorem~\ref{thm:dyckerhoff}(2) says that 
  $HN^{\Z/2}_*(\mf(f))$ is a graded free module of finite rank over the ring $k[t,t^{-1}][[u]] = k[[v]][t, t^{-1}]$ (see Example \ref{ex:technical}),
  and similarly for $HP$. In particular, Theorem~\ref{thm:dyckerhoff}(2) implies that $HN^{\Z/2}_*(\mf(f))$ is $v$-torsion free, and thus its quotient by $v$ may be identified with
$HH^{\Z/2}_*(\mf(f))$. A choice of $k[t,t^{-1}]$-linear splitting of
$HN^{\Z/2}_*(\mf(f)) \onto HH^{\Z/2}_*(\mf(f))$ determines an isomorphism
$HN^{\Z/2}_*(\mf(f)) \cong \Sigma^d \Omega_f [[v]][t,t^{-1}]$ of $k[[v]][t,t^{-1}]$-modules, 
but such an isomorphism is not  canonical.
\end{rem}

\begin{proof}
Part (1) is \cite[Theorem 6.6]{dyckerhoff}; it also follows from Theorem~\ref{thm:hkr} (and Remark~\ref{rem:hkr}). It is a consequence of (1) that the Hodge-to-de Rham spectral sequence degenerates \cite[Section 7]{dyckerhoff}; the statement in (2) concerning negative cyclic homology follows: see e.g. \cite[Proof of Proposition 9]{shklyarov1}. The statement in (2) about periodic cyclic homology is then~clear. 
\end{proof}

\section{A Wang-type exact sequence}
\label{sec:wang}

Let $k$, $Q$, $f$, and $R$ be as in Section~\ref{sec:tech}. The goal of this section is to leverage the calculations in Theorem~\ref{thm:dyckerhoff} to compute the negative cyclic and periodic cyclic homology of $\mf(f) \simeq \Dsingdg(R)$
relative to $k$ rather than $k[t, t^{-1}]$.

\subsection{The distinguished triangle}

A key tool is the following distinguished triangle, whose associated long exact sequence on periodic cyclic homology is reminiscent of the Wang exact sequence of a circle fibration: see Remark~\ref{wang}(3). This result implies Theorem~\ref{introwang} from the introduction.

  \begin{thm} 
  \label{wangthm} Let $k$ be a field of characteristic $0$, $Q$ a smooth $k$-algebra, and $f \in Q$ a non-zero divisor.

  \begin{enumerate}
  \item There is a distinguished triangle
    $$
    HN(\mf(f)) \to  HN^{\Z/2}(\mf(f)) \xra{h} HN^{\Z/2}(\mf(f)) \to
    $$
    in the derived category of  dg-$k [ u ]$-modules, where $h$ is a map satisfying 
    $$
    h(g(t) \cdot \a) = u g'(t) \a + g(t) h(\a)
    $$
    on the level of homology for any cycle $\a$ and $g \in k[t, t^{-1}]$. In more detail: in terms of the equivalent de Rham models, $h$ is given by the endomorphism
    $u \frac{\partial}{\partial t} + \l_f$ of $\Omega^\bu_Q[t, t^{-1}][[u]]$. 
    
\item Assume that the morphism $f: \Spec(Q) \to \A^1_k$ determined by $f$
has only one singular point, $\fm \in \Spec(Q)$, and it lies over the origin (i.e., $f \in \fm$). We have, for all $j \in \Z$ such that $j \equiv \dim(Q_\fm) \text{ (mod 2)}$, an exact sequence 
$$
0 \to HN_{j} (\mf(f)) \to HN^{\Z/2}_{j} (\mf(f)) \xra{h} HN^{\Z/2}_{j} (\mf(f)) \to HN_{j - 1} (\mf(f)) \to 0.
$$
\item Parts (1) and (2) hold with $HN$ replaced with $HP$ throughout. 
\end{enumerate}
  \end{thm}

  \begin{proof} 
    We have a  canonical isomorphism
    $$
    \Omega^\bu_Q[t, t^{-1}] \oplus \Omega^\bu_Q[t, t^{-1}] \xra[\cong]{(\a, \b) \mapsto \a + \b dt}    \Omega^\bu_{Q[t,t^{-1}]},
    $$
    under which the differential for $HN^{\deR}(Q[t, t^{-1}], ft)$ corresponds to     
    $$
    u \begin{bmatrix} d_Q & 0 \\ \frac{\partial}{\partial t} & d_Q \end{bmatrix}
    + \begin{bmatrix} \l_{t df} & 0 \\  \l_f & \l_{t df}   \\ \end{bmatrix};
    $$
      here, $d_Q$ denotes the de Rham differential on $\Omega^\bu_Q$, extended $k[t,t^{-1}]$-linearly to $\Omega^\bu_Q[t, t^{-1}]$.
It is thus immediate that $HN^{\deR}(Q[t,t^{-1}], ft)$ is isomorphic, as a dg-$k[[u]]$-module, to the homotopy fiber of the map
    $$
    HN^{\deR, \Z/2}(Q[t, t^{-1}], ft) \xra{u \frac{\partial}{\partial t} + \l_f}
    HN^{\deR, \Z/2}(Q[t, t^{-1}], ft). 
    $$
Letting $h \ce u \frac{\partial}{\partial t} + \l_f $, Part (1) therefore follows from Parts (2) and (3) of Theorem~\ref{thm:hkr}. Part (2) is immediate from (1) and Theorem~\ref{thm:dyckerhoff}(2), and Part (3) follows by the exactness of inverting $u$. 
\end{proof}

\begin{rem}
\label{wang}
\text{ }
\begin{enumerate}
\item As is evident from the relation $h(g(t) \cdot \a) = u g'(t) \a + g(t) h(\a)$, the map $h$ is not $k[t, t^{-1}]$-linear, although its source and target are both complexes of
  $k[t, t^{-1}]$-modules. Put differently, $HN^{\Z/2}(\mf(f))$ is two-periodic, but the map $h$ is not. 
\item  
Suppose $Q = \C[x_0, \dots, x_{n+1}]$ and that the hypersurface $Q/(f)$ has an isolated singularity at the origin. Denote by $\varphi \co E \to S^1$ the \emph{Milnor fibration} associated to $f$; see \cite[Chapter 3]{dimca92} for background on the Milnor fibration. Milnor proves in~\cite[Theorem 6.5]{milnor} that the fiber $F$ of this fibration is homotopy equivalent to a wedge sum of copies of $S^{n+1}$; $F$ is called the \emph{Milnor fiber}. The Serre spectral sequence associated to the Milnor fibration collapses to a long exact sequence, called the \emph{Wang exact sequence}, of the form
\begin{equation}
\label{wangeq}
\cdots \to H^i(E) \to H^i(F) \xra{T - \id} H^i(F) \to H^{i+1}(E) \to \cdots,
\end{equation}
where the maps $H^i(E) \to H^i(F)$ are the natural ones, and $T$ denotes the automorphism of $H^*(F)$ induced by monodromy. 

On the other hand, by a Theorem of Efimov~\cite[Theorem 1.1]{efimov}, there is a canonical isomorphism
$$
HP^{\Z/2}_j(\mf(f)) \cong \bigoplus_{i \in \Z}H^{j + 2i + 1}(F ; \C)
$$
for all $j$ (see \cite{BP} for an $\ell$-adic version of this result); we note that~\cite[Theorem 1.1]{efimov} identifies periodic cyclic homology with vanishing cohomology, but the latter is isomorphic to the cohomology of the Milnor fiber since the singularity of $Q/(f)$ is isolated. It is therefore reasonable to consider the long exact sequence in Theorem~\ref{wangthm}(3) as an analogue of the Wang exact sequence. Theorem~\ref{wangthm} also suggests that one should consider $HP_*(\mf(f)) \cong HP_*(\Dsingdg(R))$ as an analogue of the cohomology of the total space $E$ of the Milnor fibration.
\end{enumerate}
\end{rem}

\subsection{The homogeneous case}

  We now assume $Q$ is the polynomial ring $k[x_0, \dots, x_{n+1}]$, equipped with its standard internal grading given by $\deg(x_i) = 1$ for all $i$. Assume also that 
  $f$ is homogeneous, say of internal degree $\deg(f) = e$. 
Henceforth, we will use the notation $|-|$ for homological degree and $\deg(-)$ for internal degree.
The main goal of this subsection is to prove Corollary~\ref{cor:0}, which provides an important link between the $\Z$-graded and $\Z/2$-graded negative cyclic homology of the singularity category of a homogeneous isolated hypersurface singularity.

  Recall that $\Omega^\bu_Q[t, t^{-1}]$ has a homological grading, 
  where $|t| = -2$ and $|\Omega^p_Q| = p$; we also equip $\Omega^p_Q[t, t^{-1}]$ with an internal grading by declaring $\deg(t) = -e$ and
  $\deg(g_0 dg_1 \cdots dg_j) = \sum_{i=0}^j  \deg(g_i)$.
  With these conventions, we have $|ft| = -2$ and  $\deg(ft) = 0$, and the de Rham differential $d_Q$ is an operator of homological degree $+1$ and internal degree $0$. 
 
\begin{rem}
\label{rem:extend}
  It is tempting to extend the internal grading to $\Omega^\bu_Q[t,t^{-1}][[u]]$  by declaring $\deg(u) = 0$,
  so that, for instance,  the differential on $HN^{\deR, \Z/2}(Q[t, t^{-1}], ft)$, namely $u d_Q + t \l_{df}$ has internal degree $0$.
  But, this does not entirely make sense,
  for observe that this would require setting $\deg(v) = \deg(ut^{-1}) = e$, and, since $\Omega^\bu_Q[t,t^{-1}][[u]] = \Omega^\bu_Q[[v]][t,t^{-1}]$, this would mean we are ``grading''
  the power series ring $k[[v]]$ with $\deg(v) \ne 0$, which is nonsensical.
 \end{rem} 
 
  Remark~\ref{rem:extend} notwithstanding, we define a grading-like operator $\G$ on  $\Omega^\bu_Q[[v]][t, t^{-1}] = \Omega^\bu_Q[t,t^{-1}][[u]]$ by setting $\G(\omega t^iu^j) = (\deg(\omega ) - i \cdot e)\omega
  t^i u^j$ for $\omega   \in \Omega^\bu_Q$ and extending along infinite sums.
  Equivalently, and more explicitly, we define the operator $\Gamma$ on 
  $\Omega^\bu_Q [[ v ]][t, t^{-1}]$ by declaring $\Gamma(v) = ev$ and extending along infinite sums as follows: a typical element of $\Omega^\bu_Q [[ v ]][t, t^{-1}]$ has the form $\a = \sum_j \sum_m \sum_p \omega_{j,m,p} v^m t^j$, where $j$ ranges over a finite subset of $\Z$,
$m$ ranges over all non-negative integers, $p$ ranges over a finite subset of $\N$, and $\o_{j,m,p}$ is a homogeneous element of $\Omega^\bu_Q$ of internal degree $p$. We~set
$$
\Gamma(\a) = \sum_{j,m,p} (p + em-ej) \omega_{j,m,p} v^m t^j.
$$
Since $\deg(u) = 0$, $\Gamma$ is a $k[u]$-linear derivation: 
  $\Gamma(u \a) = u \G(\a)$, and $\Gamma(\a \b) = \Gamma(\a) \b + \a \Gamma(\b)$. Note that the homological grading is ignored, and there are no signs.

\begin{notation} For any integer $i$ and operator $\Gamma$ on a $k$-vector space $W$,  
  define $\Gamma_i  W \ce \ker(\Gamma-i)$, the eigenspace of $\Gamma$ with eigenvalue $i$.  
\end{notation}

  For an honest $\Z$-graded vector space $W$, with $\Gamma$ taken to be the grading operator,  $\Gamma_i(W)$
  coincides with the $i^{\th}$ graded piece of $W$, which we will typically write as $[W]_i$; we thus have $W = \bigoplus_i \Gamma_i(W) =  \bigoplus_i [W]_i$. 
  Such a decomposition does not hold for $W = \Omega^\bu_Q[[v]][t, t^{-1}]$, even if we take $Q = k$.
  For instance, $\G_{j} k [[ v ]][t,t^{-1}] = 0$ when $e \nmid j$, and
  $\G_{i \cdot e} k [[ v ]][t,t^{-1}] = k[u] t^{-i}$. We therefore have:
  $$
  \bigoplus_j \G_j k [[ v ]][t,t^{-1}] = k[v][t, t^{-1}] \subsetneq k[[v]][t, t^{-1}].
  $$
More generally, for each integer $i$, 
$$
\Gamma_i \Omega^\bu_Q [[ v ]][t,t^{-1}] = \bigoplus_{j \in \Z} \bigoplus_{m \geq 0}  \Gamma_{i-em+ej} \Omega^\bu_Q v^m t^j
= \left[\Omega^\bu_Q[v, t, t^{-1}]\right]_i.
$$

The differential $t df  + u d_Q$ on $HN^{\deR}(Q[t, t^{-1}], tf)$ commutes with $\Gamma$ (loosely speaking, the differential ``has degree zero''); it follows that $\Gamma$ induces an operator on
$HN_*(\mf^{\Z/2}(f))$, which we also write as $\Gamma$. As with $\Omega^\bu_Q [[ v ]][t,t^{-1}]$, $\Gamma$ does not induce a $\Z$-grading on $HN_*(\mf^{\Z/2}(f))$, but 
$\Gamma_0 HN_*(\mf^{\Z/2}(f))$ is a module over $k[u] = \Gamma_0 k[t, t^{-1}][[u]]$.

  \begin{thm}
  \label{thm:homotopy}
  Let $Q = k[x_0, \dots, x_{n+1}]$, equipped with the internal grading given by $\deg(x_i) = 1$, and let $f$  be a 
  homogeneous form of degree $e \ge 1$. The $k[u]$-linear endomorphism
  $h = u \frac{\partial}{\partial t} + \l_f$ of $HN^{\deR, \Z/2}(Q[t, t^{-1}], ft)$ is homotopic to the endomorphism $-\frac{v}{e}\G$.
  \end{thm}

  \begin{proof} Define $\e_Q: \Omega^\bu_Q \to \Omega^\bu_Q$ to be the map induced by the {\em Euler derivation} on $Q$ sending $g \in Q$ to $\deg(g) g$; i.e.,
  $\e_Q$ maps $\Omega^p_Q$ to $\Omega^{p-1}_Q$ by the formula
  $$
  \e_Q (g_0 dg_1 \cdots dg_p) \ce \sum_{i=1}^p (-1)^{i-1} \deg(g_i) g_0g_i dg_1 \cdots \widehat{dg_i} \cdots dg_p.
  $$
  The map $\e_Q$ has homological degree $-1$ and internal degree $0$.  
  We also denote by $\e_Q$ its $k[t, t^{-1}] [[ u ]]$-linear extension (i.e., its $k[[v]][t,t^{-1}]$-linear extension) to  $\Omega^\bu_Q [t, t^{-1}]  [[ u ]] = \Omega^\bu_Q [[v]][t, t^{-1}]$.
We have relations
  $$
  [\e_Q, t \l_{df}] = \deg(df) t\l_{f } = e t\l_{f},  \quad \text{and} \quad  [\e_Q, u d_Q] = u \G_Q;
  $$
  where $[\a, \b] \ce \a \circ \b + \b \circ \a$ for operators of odd homological degree, and $\G_Q: \Omega^\bu_Q[[v]] [t, t^{-1}]  \to \Omega^\bu_Q [[v]][t, t^{-1}]$ denotes
  the $k[[v]][t, t^{-1}]$-linear map determined by $\G_Q(\o)  = \deg(\o) \o$ for $\o \in\Omega^\bu_Q$.  Since the differential on  $HN^{\deR}(\mf^{\Z/2}(f)) = \Omega^\bu_Q [[v]][t,
  t^{-1}]$ is $u d_Q + t \l_{df}$, we may interpret $\frac{\e_Q}{et}$ as a homotopy exhibiting  that $\l_f$ and $-\frac{v}{e} \G_Q$
  are homotopic endomorphisms of  $HN^{\deR}(\mf^{\Z/2}(f))$, and hence that $h$
  and  $vt \frac{\partial}{\partial t} - \frac{v}{e} \G_Q$ are homotopic. 

It remains to show $vt \frac{\partial}{\partial t}    - \frac{v}{e} \G_Q = - \frac{v}{e} \G$. We first note that, since $v = u/t$,  we have $\frac{\partial v^j}{\partial t} = -\frac{jv^j}{t}$. Thus, for $\a = \omega v^j t^i$,  we have:
  $$
  (vt \frac{\partial}{\partial t}    - \frac{v}{e} \G_Q)(\a) = (i - j + \frac{\deg(\omega)}{e}) \omega v^{j+1} t^i
= -\frac{v}{e} (\deg(\omega) + (j - i)e) \omega v^j t^i = - \frac{v}{e} \G(\a),
$$
and this extends along all infinite sums. 
\end{proof}

\begin{cor}   \label{cor:0}
  Let $Q$ and $f$ be as in Theorem~\ref{thm:homotopy}. Assume also that $n$ is even, and that $\Proj(Q/f) \subseteq \PP^{n+1}$ is smooth. 
The canonical map
$$
HN^{\deR}(Q[t, t^{-1}], ft)\to  HN^{\deR, \Z/2}(Q[t, t^{-1}], ft),
$$
  induced by the surjection $\Omega^\bu_{Q[t,t^{-1}]/k} \onto \Omega^\bu_{Q[t,t^{-1}]/k[t,t^{-1}]}$
 that sets $dt = 0$, induces a $k[u]$-linear isomorphism 
    $$
    HN^{\deR}_{\on{even}}(Q[t, t^{-1}], tf) \cong  \Gamma_0 HN^{\deR, \Z/2}_{*}(Q[t, t^{-1}], tf).
    $$
  \end{cor}

  \begin{proof} By Theorem~\ref{thm:dyckerhoff}(2), we have $HN_i^{\Z/2}(\mf(f)) = 0$ for $i$ odd. It therefore follows from Theorem~\ref{thm:homotopy} that, for each $j \in \Z$, there is an exact sequence 
    $$
    0 \to   HN^{\deR}_{2j}(Q[t, t^{-1}], ft) \to HN^{\deR, \Z/2}_{2j}(Q[t, t^{-1}], ft) \xra{-v \frac{\G}{e}} HN^{\deR, \Z/2}_{2j}(Q[t, t^{-1}], ft)
    $$
    of $k$-vector spaces. Again by Theorem~\ref{thm:dyckerhoff}(2), the $k[[v]]$-module $HN^{\deR, \Z/2}_{* }(Q[t, t^{-1}], ft)$ is $v$-torsion free, and hence $\ker(-v \frac{\G}{e}) = \ker(\G)$.
      \end{proof}

\section{Proof of Theorem \ref{rephrase}} \label{mainsection}
We start with the following Theorem, which is a consequence of the results in Sections~\ref{sec:tech} and~\ref{sec:wang}; it encapsulates exactly what we will need from these sections for the proof of Theorem~\ref{rephrase}:

\begin{thm} \label{newtheorem}
Let $Q$ be the standard graded polynomial ring $k[x_0, \dots, x_{n+1}]$ with $n$ even,
and let $f$ be a non-zero, homogeneous element of degree $e \geq 1$. Assume that $\Proj(Q/f) \subseteq \PP^{n+1}$ is smooth. Define $\Omega_f$ to be the graded $k$-vector space 
$$
\Omega_f \ce \frac{\Omega^{n+2}_Q}{df \cdot\Omega^{n+1}_Q} \cong 
\frac{k[x_0, \dots, x_{n+1}]}{(\frac{\partial f}{\partial x_0}, \dots, \frac{\partial f}{\partial x_{n+1}})} dx_0 \cdots dx_{n+1}.
$$
  \begin{enumerate}
  \item $\bigoplus_m HN_{2m}(\mf(f))$ is a free (homologically) graded $k[u]$-module of finite rank. 
    In particular, $HN_{2m}(\mf(f)) = 0$ for $m \gg 0$,
    and multiplication by $u$ determines an isomorphism $HN_{2m+2}(\mf(f)) \xra{\cong}
    HN_{2m}(\mf(f))$ for $m \ll 0$.

  \item  
The grading operator $\Gamma$ on the de Rham HN complex $HN^{\dR, \Z/2}(Q[t, t^{-1}], ft)$ induces an operator on $HH_*^{\dR, \Z/2}(Q[t, t^{-1}], ft)$, and, for all $m \in \Z$, there is an isomorphism
  $$
  \left[\Omega_f\right]_{(\p -m) \cdot e} \xra{\cong} \Gamma_0 HH^{\dR, \Z/2}_{2m}(Q[t, t^{-1}], ft)
  $$
  induced by sending $\o \in [\Omega^{n+2}_Q]_{(\p -m) \cdot e}$ to the class $\o t^{\p-m}$.

  \item For each $m$, we have a short exact sequence 
     \begin{equation} \label{E27d}
    0 \to HN(\mf(f))_{2m+2} \xra{u} HN_{2m}(\mf(f)) \xra{q} \left[\Omega_f\right]_{(\p -m) \cdot e}
    \to 0;
         \end{equation}
    here, the map $q$ is the composition
    \begin{align*}
    HN_{2m}(\mf(f)) \xra{\cong} HN^{\dR}_{2m}(Q[t, t^{-1}], ft) & \xra{dt \mapsto 0} \Gamma_0 HN^{\dR, \Z/2}_{2m}(Q[t, t^{-1}], ft) \\ & \xra{u \mapsto 0} \Gamma_0 HH^{\dR, \Z/2}_{2m}(Q[t, t^{-1}], ft) \\
    & \xra{\cong} \left[\Omega_f\right]_{(\p -m) \cdot e},     
    \end{align*}
    where the first isomorphism is from Theorem \ref{thm:hkr}(2), and the last is (the inverse of) the isomorphism in (2). 
 \end{enumerate}
\end{thm}

\begin{proof} Part (1) follows from Theorem \ref{thm:dyckerhoff}(2) and Corollary~\ref{cor:0}, using that $\Gamma_0 (k [[ v ]][t,t^{-1}]) = k[u]$. Let us now prove (2) and (3). To ease notation, we write $V  = HN^{\dR,\Z/2}(Q[t, t^{-1}], ft)$ and $W = HH^{\dR, \Z/2}(Q[t, t^{-1}], ft)$.
Since the homologies of both $V$ and $W$ vanish in odd degrees, the distinguished triangle
$$
V[2] \xra{u \cdot} V \xra{u \mapsto 0} W \to
$$
(where $V[2]$ denotes the shift of $V$ by 2 in homological degree) yields a short exact sequence
$$
  0 \to V_{2m+2} \xra{u\cdot} V_{2m} \xra{u \mapsto 0} W_{2m} \to 0.
$$
Since $\deg(u) = 0$, the map $V \xra{u \cdot} V$ commutes with $\Gamma$; it follows that $\Gamma$ induces an operator on $W$. Moreover, this operator commutes with the differential on $W$ and therefore induces an operator on its homology. The isomorphism $\Sigma^d \Omega_f \cong HH^{\dR, \Z/2}(Q[t, t^{-1}], ft)$ arising from Theorem~\ref{thm:dyckerhoff}(1) induces the desired isomorphism 
$\left[\Omega_f\right]_{(\p -m) \cdot e} \xra{\cong} \Gamma_0 HH^{\deR, \Z/2}_{2m}(\mf(f))$, which proves (2). 

We evidently have an exact sequence
$0 \to \Gamma_0 V_{2m+2} \xra{u\cdot} \Gamma_0 V_{2m} \xra{u \mapsto 0} \Gamma_0 W_{2m}.$ It is straightforward to check that $\Gamma_0 V_{2m} \xra{u \mapsto 0} \Gamma_0 W_{2m}$ is surjective, so in fact we have a short exact sequence
$$
0 \to \Gamma_0 V_{2m+2} \xra{u\cdot} \Gamma_0 V_{2m} \xra{u \mapsto 0} \Gamma_0 W_{2m} \to 0
$$
of vector spaces. Applying Corollary \ref{cor:0} again, we obtain the short exact sequence
  $$
  0 \to HN^{\dR}_{2m+2}(Q[t, t^{-1}], ft) \xra{u\cdot}  HN^{\dR}_{2m}(Q[t, t^{-1}], ft)  \xra{u \mapsto 0} \Gamma_0 HH^{\dR,\Z/2}_{2m}(\mf(f)) \to 0.
  $$
  The square
  $$
\xymatrix{
HN_{2m+2}(\mf(f)) \ar[d]^-{\cong} \ar[r]^-{u \cdot} & HN_{2m}(\mf(f)) \ar[d]^-{\cong} \\
HN^{\dR}_{2m+2}(Q[t, t^{-1}], ft) \ar[r]^-{u \cdot} & HN^{\dR}_{2m}(Q[t, t^{-1}], ft) \\
}
  $$
  evidently commutes, where the vertical isomorphisms arise from Theorem~\ref{thm:hkr}. We therefore arrive at the short exact sequence
$$
  0 \to HN_{2m+2}(mf(f)) \xra{u\cdot}  HN_{2m}(mf(f)) \xra{u \mapsto 0} \Gamma_0 HH^{\dR,\Z/2}_{2m}(\mf(f)) \to 0.
  $$
  The exactness of \eqref{E27d} therefore follows from the commutativity of the triangle
  $$
\xymatrix{ 
HN_{2m}(mf(f)) \ar[r]^-{u \mapsto 0} \ar[rd]^-{q} & \Gamma_0 HH^{\dR,\Z/2}_{2m}(\mf(f)) \\
& [\Omega_f]_{(\frac{n}{2} +1 -m) \cdot e} \ar[u]_-{\cong},
}
  $$
  where the vertical isomorphism is given by (2).
  \end{proof}

  \begin{cor} In the setting of Theorem~\ref{newtheorem}, the $k$-vector space $HP_0(\mf(f))$ has dimension equal to
  $\dim_\C \left[\Omega_f\right]_{\Z \cdot e}$.
 Moreover, setting $F_{\nc}^p =  F_{\nc}^p HP_0(\mf(f))$, we have canonical isomorphisms
  $$
  \frac{F_{\nc}^p}{F_{\nc}^{p+1}} \cong 
  \left[ \Omega_f \right]_{\left(\p  - p\right) \cdot e}
  $$
  for each integer $p$. 
\end{cor}

\begin{rem} We are not claiming that there is a \emph{canonical} isomorphism $HP_0(\mf(f)) \cong~\left[\Omega_f\right]_{\Z \cdot e}$.
\end{rem}

\begin{rem}
It follows from the definition of a Hodge structure of weight 0 that the intersection of $\Hdg(HP_0(\mf(f)))$ with $F^1_{\nc}$ is $0$. Thus, the composition
$$
\Hdg(HP_0(\mf(f))) \into F^0_{\nc} \xra{\can} F^0_{\nc} / F^1_{\nc} \cong \left[ \Omega_f \right]_{\left(\tfrac{n+2}{2}  \right) \cdot e}
$$
is injective; in particular, the Hodge classes of $HP_0(\mf(f))$ may be identified with a rational subspace of $\left[\Omega_f\right]_{\left(\tfrac{n+2}{2}\right) \cdot e}$. As a consequence, we see that there is no information lost when passing from the Chern character map $\ch_{HN} \co K_0(\mf(f)) \to HN_0(\mf(f))$ taking values in negative cyclic homology to the \emph{a priori} coarser map $\ch_{HH} \co K_0(\mf(f)) \to HH_0(\mf(f))$ given by the composition $K_0(\mf(f)) \xra{\ch_{HN}} HN_0(\mf(f)) \xra{u \mapsto 0} HH_0(\mf(f))$. 
\end{rem}

\subsection{The commutative diagram}
\label{sec:(2)}

We now prove Theorem~\ref{rephrase}(2). This follows from the existence and properties of the  diagram
\begin{equation} \label{E71}
\xymatrix{
K_0(X)\ar[d] \ar[r]^-{=} & K_0(X)\ar[d] \ar[r] & K_0(\Dsingdg(R)) \ar[d]  \\
KU^0(X)\ar[d]^-{\ch^{\top}}  \ar[r]^-{\cong} & K^{\top}_0(X) \ar[d]^-{\ch^{\top}} \ar@{->>}[r] & K^{\top}_0(\Dsingdg(R)) \ar[d]^-{\ch^{\top}}  \\
H^{\on{even}}(X;\C )  \ar[r]^-{\cong} & HP_0(X) \ar@{->>}[r] & HP_0(\Dsingdg(R)).\\
}
\end{equation}
The top vertical maps are the canonical ones. 
The map $KU^0(X) \to K^{\top}_0(X)$ is Blanc's comparison isomorphism, 
and the bottom horizontal map on the left is the HKR isomorphism.
Letting $E$ denote $K$, $K^{\top}$ or $HP$, 
the horizontal maps on the right side are defined by the sequence maps
\begin{equation} \label{E72}
  E_0(X) \xora{p^*}  E_0(W) \xla{\cong} E_0(\Db(R)) \xra{\can} E_0(\Dsingdg(R))
\end{equation}
where $W = \Spec(R)$, $p: W \to X$ is the map given by modding out the $\C^*$ action, and the isomorphism is given by Proposition \ref{point}.
(The fact that $p^*$ is indeed surjective, as indicated, will be justified below.)

In particular, diagram \eqref{thediagram} is the  ``boundary'' of diagram \eqref{E71}. It therefore suffices to prove \eqref{E71} commutes,
the two maps 
$K_0^{\top}(X) \to K_0^{\top}(\Dsingdg(R))$ 
and $HP_0(X) \to HP_0(\Dsingdg(R))$ are surjective as indicated, and the images of $K_0(X) \to K_0^{\top}(\Dsingdg(R))$ and $K_0(\Dsingdg(R)) \to K_0^{\top}(\Dsingdg(R))$
coincide. 

The commutativity of the top left square of \eqref{E71} is a consequence of the
construction of Blanc's map, and the bottom left square commutes by~\cite[Proposition 4.32]{blanc}.
The right side of this diagram commutes by the naturality of the map from algebraic to
topological $K$-theory and the topological Chern character map. Let us now justify that $p^*: E_0(X) \to E_0(W)$ is onto for each of $E = K$, $K^{\top}$ or $HP$. 
Toward this goal, let $Y$ be the blow-up of $\Spec(R)$ at $\fm$. 
The fiber of this blow-up is $X$, and the inclusion $i: X \into Y$ is the zero section of a map $\pi: Y \to X$ making $Y$ into a line
bundle over $X$. Moreover, we may identify $W$ with $Y \setminus X$; let $j: W \into Y$ denote the canonical open immersion.
Then we have $p = \pi \circ j$, and since $\pi$ is a line bundle over a smooth base, 
$\pi^*: E_0(X) \xra{\cong} E_0(Y)$ is an isomorphism. 
Since $Y$, $W$ and $X$ are all smooth, d\'evissage gives that 
$j^*$ fits into the long exact sequence
$$
\cdots \to E_0(Y) \xra{j^*} E_0(W) \to E_{-1}(X) \to \cdots.
$$
For each of these functors, we have $E_{-1}(X) = 0$, and thus $p^*$ is surjective, as claimed. Now assume $E = K^{\top}$ or $E = HP$. 
Proposition \ref{HPsing} gives that $E_0(\Db(R)) \to E_0(\Dsingdg(R))$ is onto, and thus the right-most map in \eqref{E72} is also surjective in these two cases.
This proves the lower two horizontal maps on the right side of \eqref{E71} are surjections as indicated.

To complete the proof, it suffices to show the image of the map 
$$
K_0(\Dsingdg(R)) \to  K_0^{\top}(\Dsingdg(R))
$$
and the image of the composition
$$
G_0(R) = K_0(\Db(R)) \to  K_0(\Dsingdg(R)) \to K_0^{\top}(\Dsingdg(R))
$$
coincide.

\begin{rem}
The map 
  $K_0(\Db(R)) \to K_0(\Dsingdg(R))$ itself need not be onto, due to the fact that $K_{-1}(R)$ is typically non-zero; see \cite{CHWW}.
  \end{rem}
However, the map 
  $K_0(\Db(R)) \to K_0(\Dsingdg(R))$ is onto ``up to $\A^1$-homotopy''. In detail,
consider the diagram
$$
\xymatrix{
  K_0(\Db(R[x])) \ar[r] \ar[d]^{i^*_0 - i^*_1} & K_0(\Dsingdg(R[x]) \ar[r] \ar[d]^{i^*_0 - i^*_1}  &   K_{-1}(\Perf(R[x])) \ar[d]^{i^*_0 - i^*_1}  \ar[r] & 0 \\
  K_0(\Db(R)) \ar[r]   & K_0(\Dsingdg(R)) \ar[r] &   K_{-1}(\Perf(R)) \ar[r]  & 0  \\
  }
  $$
  with exact rows and in which the vertical maps are given by the difference of the two maps induced by setting $x$ equal to $0$ and $1$. (We may model $\Db$ as
  bounded below complexes of finitely generated projective modules with bounded homology, and with this model it is clear that setting $x$ equal to any constant determines a dg-functor. This restricts to a dg-functor on $\Perf$ and hence on $\Dsingdg$.)  Let us write $\ov{K_0(-)}$ for the cokernels of the columns of this diagram, so that we have a right
  exact sequence
  \begin{equation} \label{E73}
\ov{K_0(\Db(R))} \to \ov{K_0(\Dsingdg(R))} \to \ov{K_{-1}(\Perf(R))}\to 0.
  \end{equation}
The result we seek follows directly from the following two claims:
  (1) the map $K_0(\Dsingdg(R)) \to K_0^{\top}(\Dsingdg(R))$ factors through the canonical
  surjection
  $K_0(\Dsingdg(R)) \onto   \ov{K_0(\Dsingdg(R))}$, and (2) the map
$\ov{K_0(\Db(R))} \to \ov{K_0(\Dsingdg(R))}$ is onto. The first claim follows from the fact that $K^{\top}$ is $\A^1$-homotopy invariant. For the second claim, since the functor
$\ov{K_{-1}(-)}$ is $\A^1$-homotopy invariant and $R$ is standard graded, we have $\ov{K_{-1}(R)} \cong K_{-1}(k) = 0$. The second claim therefore follows from \eqref{E73}.

\subsection{An alternative description of the map $\alpha$}

We next establish in Lemma~\ref{lem25a} an alternative description of the map $\alpha: H^n_{\prim}(X) \to HP_0(\Dsingdg(R))$ defined in \eqref{alphadef}. 
This description will be used to show that it is an isomorphism that preserves Hodge filtrations.

We begin with some setup. Let $U$ denote the affine variety $\PP^{n+1} \setminus X$. Applying the distinguished triangle \eqref{eqn:HPdevtri}, we obtain a d\'evissage long exact sequence
\begin{equation}
\label{eqn:devUX}
\cdots \to HP_1(\PP^{n+1}) \to HP_1(U) \xra{\del_{U,X}} HP_0(X) \to HP_0(\PP^{n+1}) \to \cdots
\end{equation}
(the subscript $U,X$ on the boundary map is included to distinguish it from the other boundary maps we consider). We set
$$
HP_0^{\prim}(X) \ce \ker(HP_0(X) \to HP_0(\PP^{n+1})).
$$
Since $X$ is a smooth hypersurface in $\PP^{n+1}$ of even dimension, we have $
H^n_{\prim}(X; \C) \cong HP_0^{\prim}(X).
$
As $HP_1(\PP^{n+1}) = 0$, it follows that there is an isomorphism
\begin{equation} \label{E17a}
\del_{U,X}: HP_1(U) \xra{\cong} HP^{\prim}_0(X).
\end{equation}
Let $V = \Spec(Q[\tfrac{1}{f}])$, the open complement of $\Spec(R) = \Spec(Q/f)$ in $\bA^{n+2} = \Spec(Q)$.
There is a canonical surjection $p: V \onto U$ given by modding out the action of $\C^*$, and it induces a map
$$
p^*: HP_1(U) \to HP_1(V).
$$
We also have the d\'evissage long exact sequence
\begin{equation}
\label{eqn:devVR}
\cdots \to HP_1(\bA^{n+2}) \to HP_1(V) \xra{\del_{V,R}} HP^{\BM}_0(R) \to HP_0(\bA^{n+2}) \to  HP_0(V) \to \cdots.
\end{equation}
Since $HP_1(\bA^{n+2}) = 0$, and $HP_0(\C) \cong HP_0(\bA^{n+2}) \to  HP_0(V)$ is  injective, the boundary map $\del_{V,R}$ is an isomorphism.

\begin{lem} \label{lem25a} The composition 
  $$
  H^n_{\prim}(X, \C) \cong HP_0^{\prim}(X) \xra[\cong]{\del_{U,X}^{-1}} HP_1(U) \xra{p^*} HP_1(V) \xra[\cong]{\del_{V,R}} HP_0^{\BM}(R) \onto HP_0(\Dsingdg(R))
  $$
  coincides with the map $\alpha$ defined in \eqref{alphadef}.
\end{lem}

\begin{proof} This is a diagram chase involving the long exact sequences~\eqref{eqn:devUX} and~\eqref{eqn:devVR}, as well as the d\'evissage long exact sequence
\begin{equation}
\label{eqn:devVW}
\cdots \to HP_1(V) \xra{\del_{V,W}} HP_0(W) \to HP_0(\A^{n+2} \setminus \{0\}) \to HP_0(V) \to \cdots,
\end{equation}
where, as above, $W = \Spec(R) \setminus \{\fm\}$. 
  In a bit more detail: the naturality of these d\'evissage sequences, along with the fact that \eqref{eqn:devVW} maps to both \eqref{eqn:devUX} and \eqref{eqn:devVR}, yields the commutative  diagrams
  $$
\xymatrix{
    HP_1(U) \ar[r]^{p^*} \ar[d]^{\del_{U,X}}  & HP_1(V) \ar[d]^{\del_{V,X}} \\
    HP_0(X) \ar[r]^{p^*} & HP_0(W)
}
  \and
  \xymatrix{
    HP_1(V) \ar[d]^{\del_{V,W}} \ar[dr]^{\del_{V, R}} & \\
    HP_0(W) & HP_0^{\BM}(R). \ar[l]^\cong \\
    }
  $$
  The statement follows.
\end{proof}

\begin{prop} \label{prop17a} The map $\alpha$ is an isomorphism.
\end{prop}

\begin{proof}  
By Theorem \ref{thm:hkr} and Lemma \ref{lem25a}, it suffices to prove that the composition 
  \begin{equation} \label{E522c}
  HP^{\deR}_1(U) \xra{p^*} HP^{\deR}_1(V) \xra[\cong]{\del_{V,R}^{\deR}} HP_0^{\deR}(Q[t], ft)  \onto HP^{\deR} _0(Q[t, t^{-1}], ft)
  \end{equation}
  is an isomorphism, where $\del_{V,R}^{\deR}$ is the de Rham version of the boundary map $\del_{V,R}$ in \eqref{eqn:devVR}. By Proposition~\ref{HPsing}, the last map in this composition is surjective, and its kernel is given by the image of $\C \xra{\cong} HP^{\deR}_0(R) \to HP^{\deR}_0(Q[t], ft)$: we begin by identifying this image. We have a commutative diagram
$$
\xymatrix{
\Z \ar[d] \ar[r]^-{\cong}_-{1 \mapsto [R]} & K_0(R) \ar[d]^-{\ch_{HP}} \ar[r]^-{} & G_0(R) \ar[d]^-{\ch_{HP}} \\
\C \ar[r]^-{\cong} & HP_0^{\deR}(R) \ar[r] & HP_0^{\deR}(Q[t], tf),
}
$$
where the right-most horizontal maps are induced by the inclusion $\Perf(R) \into \Db(R)$, and the left-most vertical map is the inclusion. The image of $1 \in \C$ under the composition $\C \xra{\cong} HP^{\deR}_0(R) \to HP^{\deR}_0(Q[t], ft)$ is therefore $\ch_{HP}([R])$, which is equal to $[dfdt]$ by \cite[Example 6.4]{BW}. 

The formula for $\del^{\deR}_{V,R}$ given in Theorem \ref{deRhamModels}  implies that $\del^{\deR}_{V,R}([df/f]) = [df  dt]$. We thus need only show that
the map
\begin{equation} \label{E522b}
HP^{\deR}_1(U) \oplus \C \xra{\left(p^*, \left[\tfrac{df}{f}\right]\right)}
HP^{\deR}_1(V)
\end{equation}
where  $\left[\frac{df}{f}\right] \in HP^{\deR}_1(V)$, is an isomorphism. This appears to be well-known (see e.g. \cite[Chapter 6, Section 1]{dimca92}), but we sketch a proof. 

Set $A \coloneqq \C[x_0, \dots, x_{n+1}][\frac{1}{f}]$, and recall that $V = \Spec(A)$ and $U = \Spec(A_0)$.
The Euler derivation gives a contracting homotopy on the internal degree $j$ part of the de Rham complex $(\Omega^*_A, d)$
for all $j \ne 0$, and thus we may identify the de Rham cohomology of $V$ with the cohomology of the complex $([\Omega^*_A]_0, d)$.
Moreover, we have an isomorphism 
\begin{equation} \label{E522}
\Omega^*_{A_0} \oplus \Omega^{*-1}_{A_0} \xra{\cong} \left[\Omega^*_A\right]_0
\end{equation}
given by $(\a, \b) \mapsto \a + \frac{df}{f} \b$.
This gives an isomorphism 
$$
(p^*, [\tfrac{df}{f}] p^*): H^m_{\deR}(U) \oplus H^{m-1}_{\deR}(U) \xra{\cong}  H^m_{\deR}(V) 
$$
for each $m$. The isomorphism \eqref{E522b} thus follows from the HKR isomorphisms
$HP^{\deR}_1(V) \cong H_{\deR}^{\odd}(V)$
and
$HP^{\deR}_0(U) \cong H_{\deR}^{\even}(U)$, along with the fact that,
since $U$ is the complement of a smooth projective  hypersurface of even dimension, we have 
$H_{\deR}^{\even}(U) = H_{\deR}^0(U) = \C$.
\end{proof}

\subsection{Spanning set for $HN^{\deR}_{2m}(Q[t, t^{-1}], ft)$}

Fix $m \in \Z$. We next exhibit an explicit spanning set for $HN^{\deR}_{2m}(Q[t, t^{-1}], ft)$ as a complex vector space. This is the content of Lemma~\ref{lem:psi}, which plays a key role in the identification of the ``polar filtration" on $HP_1(U)$ and the nc Hodge filtration on $HP_0(\D^{\on{sg}}(R))$ (Lemma~\ref{lem727}). If $j \leq \tfrac{n+2}{2} - m$, then
given $\o \in [\Omega^{n+2}_Q]_{j \cdot \deg(f)}$, we define 
\begin{equation} \label{E211}
  \psi_{m,j}(\o) \coloneqq   \frac{(-1)^j}{j!} d(\e_Q(\o) t^j) u^{\tfrac{n+2}{2} -m-j} \in HN^{\deR}(Q[t, t^{-1}], ft),
\end{equation}
where $\e_Q$ is as defined in the proof of Theorem~\ref{thm:homotopy}.
Observe that $\psi_{m,j}(\o)$ has homological degree $2m$ and internal degree $0$. We have
$$
d_Q(\e_Q(\o)) = d_Q(\e_Q(\o)) + \e_Q(d_Q(\o))  =  \deg(\o) \o = j \deg(f) \o,
$$
where the first equality holds since $d_Q(\o) = 0$, and the second follows from the proof of Theorem~\ref{thm:homotopy}. We may therefore equivalently write 
\begin{equation}
\label{eqn:equiv}
\psi_{m,j}(\o) =  \frac{(-1)^j \deg(f)}{(j-1)!} \left( \o t^j  -  \tfrac{\e_Q(\o)}{\deg(f)} t^{j-1} dt \right) u^{\tfrac{n+2}{2} -m- j}.
\end{equation}
Using \eqref{eqn:equiv} along with Euler's formula $\deg(f) \cdot f = \sum_{i = 0}^{n+1} \frac{\partial f}{\partial x_i} x_i$, one checks that $\psi_{m,j}(\o)$ is a cycle, and so it determines a class in 
$HN_{2m}^\deR(Q[t, t^{-1}], ft)$. We  write $\psi_{m}$ for the induced map
\begin{equation} \label{E211b}
\psi_m: \bigoplus\limits_{j \leq \tfrac{n+2}2 - m} \left[\Omega^{n+2}_Q\right]_{j \cdot \deg(f)} \to HN_{2m}^\deR(Q[t, t^{-1}], ft).
\end{equation}
Setting $m = 0$, and replacing $HN$ with $HP$, we obtain  the map
\begin{equation} \label{E211c}
  \psi: \left[\Omega^{n+2}_Q\right]_{\Z \cdot \deg(f)}
\coloneqq \bigoplus\limits_{j \in \Z} \left[\Omega^{n+2}_Q\right]_{j \cdot \deg(f)}
  \to HP_0^\deR(Q[t, t^{-1}], ft)
\end{equation}
given by the same formula: if $\o \in \left[\Omega^{n+2}_Q\right]_{j \cdot \deg(f)}$, then $\psi(\o)$ is the class of $\frac{(-1)^j}{j!} d(\e_Q(\o) t^j) u^{\tfrac{n+2}{2} -j}$.

\begin{rem}
\label{rem:colimit}
The composition of $\psi_m$ with
$$
HN_{2m}^{\deR}(Q[t,t^{-1}], ft) \xra{\can} HP_{2m}^{\deR}(Q[t,t^{-1}], ft) \xra{u^m} HP_0^{\deR}(Q[t,t^{-1}], ft)
$$
coincides with the restriction of $\psi$ to $\bigoplus\limits_{j \leq \tfrac{n+2}2 - m} \left[\Omega^{n+2}_Q\right]_{j \cdot \deg(f)}$. 
\end{rem}

\begin{lem}
\label{lem:psi}
  The map $\psi_m$ in \eqref{E211b} is a surjection for all $m$, and the map $\psi$ in \eqref{E211c} is a surjection. 
\end{lem}

\begin{proof}
To ease notation, we write $HN_*$ for $HN_*^{\deR}(Q[t, t^{-1}], ft)$. Let $e = \deg(f)$ and $p = \tfrac{n+2}{2}$, and consider the diagram
$$
\begin{tikzcd}
  0 \arrow[d] & 0 \arrow[d] \\
  \bigoplus\limits_{0 \leq j \leq p - m-1} \left[\Omega^{n+2}_Q\right]_{j \cdot e} \arrow[d,"\can"] \arrow[r,"\psi_{m+1}"] &
  HN_{2m+2} \arrow[d, "u \cdot"] \\
   \bigoplus\limits_{0 \leq j \leq p - m}  \left[\Omega^{n+2}_Q\right]_{j \cdot e}  \arrow[d, "\can"] \arrow[r,"\psi_{m}"] &
   HN_{2m} \arrow[d, "q"]   \\
\left[\Omega^{n+2}_Q\right]_{\left(p  - m\right) \cdot e}  \arrow[d]  \arrow[r,twoheadrightarrow, "\gamma"] &
 \left[\Omega_f\right]_{\left(p-m\right) \cdot e}, \arrow[d] \\
0 & 0 \\
\end{tikzcd}
$$
where the right column is given by Theorem~\ref{newtheorem}(2), and the map $\gamma$ will be defined shortly. 
The formula for $\psi_m$ and the description of $q$ in Theorem~\ref{newtheorem} imply that
$$
q(\psi_m(\o)) =
\begin{cases}
  \frac{(-1)^{p-m}\deg(f)}{(p-m-1)!} \overline{\o}, &  \text{if $|\o| = (p-m)\deg(f)$; and} \\
  0, &  \text{otherwise.} \\
\end{cases}
$$
Setting $\gamma(\o) =   \frac{(-1)^{p-m}\deg(f)}{(p-m-1)!} \overline{\o}$ thus
makes the diagram commute, and it is clear that $\gamma$ is a surjection. 
Since $HN_{2m} = 0$ for $m \gg 0$ by Theorem \ref{newtheorem}(1), it follows by descending induction that 
$\psi_m$ is surjective for all $m$. It follows from Remark~\ref{rem:colimit} that the map $\psi$ is therefore also surjective.
\end{proof}

\subsection{Relating filtrations}

The goal of this section is to relate the ``polar filtration'' on $HP_1(U)$ (defined below) with the nc Hodge filtration on $HP_0(\Dsingdg(R))$; this leads quickly to a proof that the isomorphism $\alpha$ preserves filtrations. Recall that $X = \Proj(Q/f) \subseteq \PP^{n+1}$, and $U$ denotes the affine variety $\PP^{n+1} \setminus X$. As above, we also set $V = \Spec(Q[1/f])$. The map $p \co V \to U$ given by modding out by the $\C^*$-action on $V$ induces a map $\Omega^\bullet_{U} \to \Omega^\bullet_{V}$ via pullback. As explained in \cite[Chapter 6, Section 1]{dimca92}, $p^*$ induces a chain isomorphism from $\Omega^\bullet_{U}$ to a subcomplex of $\Omega^\bullet_{V}$; in more detail, we have:
\begin{equation}
\label{eqn:UV}
\Omega^j_{U} \xra{\cong} \{\frac{\e_Q(\a)}{f^s} \in \Omega^j_{V} \text{ : } s \ge 0 \text{, } \a \in [\Omega^{j+1}_Q]_{s \deg(f)} \}
\end{equation}
    for all $j$, where $\e_Q$ is as defined in the proof of Theorem~\ref{thm:homotopy}. Given $\o \in \Omega^j_U$, we let $\on{ord}(\o)$ denote the minimum $s$ such that there is a representation of $\o$ of the form $\frac{\e_Q(\a)}{f^s}$ as above. 

\begin{defn}[\cite{dimca92} Chapter 6, Definition 1.28]
The \emph{polar filtration} on $\Omega^{\bullet}_U$ is given by
$$
P^s \Omega^i_U \ce
\begin{cases}
\{ \o \in \Omega^i_U \text{ : } \on{ord}(\o) \le i-s + 1\}, & i - s + 1 \ge 0; \\
0, & i - s + 1 < 0.
\end{cases}
$$ The polar filtration induces a filtration $P^\bullet HP_*^{\deR}(U)$ on homology in the evident way.
\end{defn}
  
    Let $\phi$ denote the composition
  $$
  \left[\Omega^{n+2}_Q\right]_{\Z \cdot \deg(f)}  \to HP^{\deR}_1(V) \to HP^{\deR}_1(U)
  $$
  given by sending $\o \in \left[\Omega^{n+2}_Q\right]_{j \cdot \deg(f)}$ to $\frac{\e_Q(\o)}{f^j} u^{\novertwo} \in HP^{\deR}_1(V)$ and then applying the isomorphism~\eqref{eqn:UV}.
Define
$$
\beta: HP^{\deR}_1(U) \to HP_0^\deR(\Db(R))
$$
to be the composition
$$
HP^\deR_1(U) \xra{p^*} HP^\deR_1(V) \xra{\del^\deR_{V,R}}  HP_0^\deR(Q[t], ft)
\xra{\can}  HP_0^\deR(Q[t, t^{-1}], ft),
$$
where $\del_{V,R}^\deR$ is the de Rham version of the boundary map $\del_{V, R}$ in the d\'evissage long exact sequence \eqref{eqn:devVR}, and $\on{can}$ is the canonical map.

\begin{lem} \label{lem727}
The diagram 
$$    
\begin{tikzcd}
  &   \left[\Omega^{n+2}_Q\right]_{\Z \cdot \deg(f)} \arrow[dl, twoheadrightarrow, "\phi"'] \arrow[dr, twoheadrightarrow, "\psi"]  \\
  HP^\deR_1(U) \arrow[rr, "\cong"', "\beta"] && HP_0^\deR(Q[t, t^{-1}], ft)
\end{tikzcd}
 $$   
commutes, $\phi$ and $\psi$ are surjective, and $\beta$ is an isomorphism. 
Moreover, $\beta$ induces an isomorphism 
$$
 P^{s + \frac{n}{2} + 1} HP_1^{\deR}(U) \xra{\cong} F_{\on{nc}}^{s} HP_0^{\deR}(Q[t, t^{-1}], ft).
$$
\end{lem}

\begin{rem}
It follows from a Theorem of Griffiths \cite[(8.6)]{griffiths} (see also \cite[Chapter 6, Section 1]{dimca92}) that $\phi$ is surjective; our argument in the proof of Lemma~\ref{lem727} gives a new proof of this fact.
\end{rem}

\begin{proof} 
  Let $\o \in \Omega^{n+1}_Q$, where $|\o| = j \cdot \deg(f)$. We have
  $$
(\beta \circ \phi)(\omega) = (\on{can} \circ \del^{\deR}_{V,R})(\frac{\e(\o)}{f^j} u^{\novertwo}) = \frac{(-1)^{j}}{j!} d(\e(\o)t^j) u^{\tfrac{n}{2} + 1 -j} = \psi(\o),
  $$
where the first and third equalities follow immediately from the definitions of $\phi$ and $\psi$, and the second is a consequence of Theorem \ref{deRhamModels}. Thus, the diagram commutes. It follows from Lemma~\ref{lem25a} and Proposition~\ref{prop17a} that $\beta$ is an isomorphism. The map $\psi$ is surjective by Lemma~\ref{lem:psi}, and so $\phi$ is surjective as well. Finally, suppose $y \in P^{s + \frac{n}{2} + 1} HP_1^{\deR}(U)$, so that $\on{ord}(y) \le \frac{n}{2} + 1 - s$. Choose $\o \in \left[\Omega^{n+2}_Q\right]_{\on{ord}(y) \cdot \deg(f)}$ such that $\phi(\o) = y$. We have:
$$
\beta(y) = \psi(\o) \in F_{\on{nc}}^{\frac{n}{2} + 1 - \on{ord}(y)} (HP_0^{\deR}(Q[t, t^{-1}], ft) \subseteq F_{\on{nc}}^{s} HP_0^{\deR}(Q[t, t^{-1}], ft).
$$
This shows that $\beta$ maps $P^{s + \frac{n}{2} + 1} HP_1^{\deR}(U)$ to $F_{\on{nc}}^{s} HP_0^{\deR}(Q[t, t^{-1}], ft)$, and a similar argument shows that $\beta^{-1}$ maps $F_{\on{nc}}^{s} HP_0^{\deR}(Q[t, t^{-1}], ft)$ to $P^{s + \frac{n}{2} + 1} HP_1^{\deR}(U)$.
\end{proof}

\subsection{Completion of the proof}

\begin{proof}[Proof of Theorem \ref{rephrase}]
The first two conditions in Properties~\ref{assumptions} follow from Part (3) of the theorem, and the third is a consequence of \cite[Theorem B]{khan} (see also \cite[Theorem 1.4]{devissage}). We proved (2) in Subsection~\ref{sec:(2)}; it therefore remains to prove (3). By Proposition~\ref{prop17a}, the map $\a$ is an isomorphism. The commutativity of \eqref{thediagram} and the surjectivity of \eqref{Ktopcomposition} imply that
$\alpha$ identifies rational structures, so we need only show that $\alpha$ induces an isomorphism
$$
F^{s} H^n_{\on{prim}}(X; \C(\frac{n}{2})) \xra{\cong} F^s_{\on{nc}} HP_0(\Dsingdg(R))
$$
 for all $s \in \Z$. It is a Theorem of Griffiths~\cite[(8.6)]{griffiths} (see also \cite[Chapter 6, Section 1]{dimca92}) that there is an isomorphism 
 $$
P^{s} H^{i}(U; \C) \cong F^s H^i(U ; \C)
$$
for all $i, s \in \Z$. We therefore have a chain of isomorphisms
$$
F^{s} H^n_{\on{prim}}(X; \C(\frac{n}{2})) \xla{\cong} F^{s} H^{n+1}(U ; \C(\frac{n}{2} + 1)) \xra{\cong} P^{s } H^{n+1}(U; \C(\frac{n}{2} + 1)) \xra{\cong} P^{s + \frac{n}{2} + 1}HP^{\deR}_1(U);
$$
the first is the boundary map in the evident long exact sequence, and the third is the identification of singular and de Rham cohomology. Applying Lemma~\ref{lem727} therefore finishes the proof. 
\end{proof}

\section{Examples}
\label{examples}

Let $R = \C[x_0, \dots, x_{n+1}]/(f)$, where $f$ is a homogeneous polynomial such that the projective hypersurface $X = \Proj(R) \subseteq \PP^{n+1}$ is smooth, and assume $n$ is even. By Theorem~\ref{rephrase}(3), the dg-category $\Dsingdg(R)$ satisfies the nc Hodge condition if and only if the Hodge Conjecture holds for $X$. In this section, we study the Hodge classes in $HP_0(\Dsingdg(R))$ in several cases in which the Hodge conjecture holds for $X$. Let us start with the simplest example:

\begin{ex}[The $n = 0$ case] In this case, $X$ is a collection of points, and so the Hodge Conjecture clearly holds for $X$.
The complexified Chern character map $K_0(X)_\C \to H^{\on{even}}(X ; \C)$ is surjective, and so the same is true of $\ch_{HP} : K_0(\Dsingdg(R))_\C \to HP_0(\Dsingdg(R))$; in other words, $HP_0(\Dsingdg(R))$ is spanned by Hodge classes. Write $f = \ell_1 \cdots \ell_d$, where each $\ell_i$ is homogeneous of degree 1, and let $M_i$ denote the $R$-module $\C[x_0, x_1] / (\ell_i)$. It is not hard to check that $K_0(\Dsingdg(R))$ is generated by $[M_1], \dots, [M_d]$ modulo the relation $\sum_{i = 1}^d [M_i] = 0$. It follows that $HP_0(\Dsingdg(R))$ is generated by $\ch_{HP}[M_1], \dots, \ch_{HP}[M_d]$ modulo the analogous relation. 
\end{ex}

Before we consider more complicated examples, we must discuss some background on Chern characters of matrix factorizations. 
\subsection{Chern characters of matrix factorizations}

It follows from the calculations in \cite[Example 6.1]{BW} that the Chern character map
$$
\ch_{HP} \co K_0(\mf(f)) \to HP_0(\mf(f)) \cong HP_0^{\dR}(Q[t, t^{-1}], tf)
$$
sends a class of the form\footnote{We recall that $K_0(\mf(f))$ is, by definition, the free abelian group generated by isomorphism classes of objects in the \emph{idempotent completion} of $\mf(f)$ modulo relations arising from exact triangles. In particular, not every class in $K_0(\mf(f))$ is necessarily of the form $[(A, B)]$, where $(A, B)$ is a matrix factorization of $f$.} $[(A, B)] \in K_0(\mf(f))$ to the class 
$$
\frac{2t^{\frac{n+2}{2}}}{(n+2)!}\on{tr}((dAdB)^{\frac{n+2}{2}} ) \in HP_0^{\dR}(Q[t, t^{-1}], tf),
$$
where $dA$ and $dB$ denote the square matrices with entries in $\Omega^1_{Q}$ obtained by applying the de Rham differential $d$ to the entries of $A$ and $B$. We note that \cite[Example 6.1]{BW} concerns Chern characters taking values in negative cyclic homology relative to $k[t,t^{-1}]$ rather than $k$, but the exact same calculations yield the above formula in our setting. 

The Chern character map is compatible with tensor products of matrix factorizations; let us explain what we mean by this. Suppose $A$ and $A'$ are $k$-algebras, $g \in A$, and $g' \in A'$. If $F \in \mf(g)$, and $F' \in \mf(g')$, then we may form the tensor product $F \otimes_k F' \in \mf(f \otimes 1 + 1 \otimes f')$; see e.g. \cite{yoshinotensor} for details. Assume now that $A = k[y_0, \dots, y_{m+1}]$ and $A' = k[y_0', \dots, y'_{m'+1}]$, $g$ and $g'$ are homogeneous, and the hypersurfaces $A/(g)$ and $A'/(g')$ both have isolated singularities. In this case, we identify $g \otimes 1 + 1 \otimes g'$ with $g + g' \in k[y_0, \dots, y_{m+1}, y'_0, \dots, y'_{m'+1}]$. The tensor product functor induces a map
$$
HP_0^{\dR}(\mf(g)) \otimes_k HP^{\dR}_0(\mf(g')) \to HP^{\dR}_0(\mf(g + g'))
$$
given by multiplication, which we denote by $\gamma \otimes \gamma' \mapsto \gamma \cdot \gamma'$. A straightforward calculation shows that $\ch_{HP}(F) \cdot \ch_{HP}(F') = \ch_{HP}(F \otimes F')$. 

\subsection{Hodge classes of Fermat hypersurfaces}
Assume now that $n \ge 2$, and suppose $f = x_0^m + \cdots + x_{n+1}^m$, so that $X$ is a Fermat hypersurface. For the remainder of this section, we will write $X$ as $X^n_m$. It follows from work of Shioda~\cite{shioda} that the Hodge classes in $H^n_{\prim}(X ; \C)$ can be explicitly described in the following way. Let $\mu_m$ denote the group of $m^{\th}$ roots of unity and $G$ the quotient of $\mu_m^{n+2}$ by the diagonal subgroup. Let $\widehat{G} = \Hom(G, \C^*)$; we identify $\widehat{G}$ with the group
$$
\{(a_0, \dots, a_{n+1}) \in (\Z/m)^{n+2} \text{ : } \sum_{i = 0}^{n+1} a_i = 0\}
$$
via the isomorphism described in~\cite[Section 1]{shioda}. We fix the following notation:
\begin{itemize}
\item $U = \{(a_0, \dots, a_{n+1}) \in \widehat{G} \text{ : } a_i \ne 0 \text{ for all } i \}$.
\item For $a \in \Z/m$, we let $\langle a \rangle$ denote the unique representative of $a$ between 0 and $m - 1$.
\item For $\a = (a_0, \dots, a_{n+1}) \in U^n_m$, we set $|\a| = \sum_{i = 0}^{n+1} \frac{\langle a_i \rangle}{m}$.
\item $B = \{\a \in U \text{ : } |t\a| = \frac{n}{2} + 1 \text{ for all } t \in (\Z/m)^\times\}$.
\end{itemize}
The group $G$ acts on $X^n_m$ by scaling the variables, and so $G$ acts on $H^n_{\prim}(X ; \C)$ as well. Given $\a \in \widehat{G}$, define
$$
V(\a) = \{ \xi \in H^n_{\prim}(X^n_m) \text{ : } g^*(\xi) = \a(g)\xi \text{ for all } g \in G\}.
$$
The following calculation of the Hodge classes of $X^n_m$ is due to Shioda:
\begin{thm}[\cite{shioda} Theorem I]
Given $\a \in \widehat{G}$, we have $\dim_\C V(\a) = 0$ or $1$, and $\dim_\C V(\a) = 1$ if and only if $\a \in U$. Moreover, the complexified 
Hodge classes of $X^n_m$ may be described as follows: 
$$
\Hdg(X^n_m) \otimes_\Q \C = \bigoplus_{\a \in B} V(\a).
$$
In particular, $\dim_\Q \Hdg(X^n_m) = |B|$. 
\end{thm}

Shioda applies this result in \cite{shioda} to confirm a family of cases of the Hodge Conjecture for Fermat hypersurfaces; see also the recent work of da~Silva~\cite{dasilva}.  

\begin{ex}
\label{ex:quadratic}
If $m = 2$, then $B = \{(1, \dots, 1)\}$, and so $\dim_\Q \Hdg(X^n_2) = 1$.  On the other hand, by Kn\"orrer periodicity, the dg-category $\Dsingdg(R)$ has exactly one indecomposable object up to homotopy equivalence, namely the tensor product of the matrix factorization $(x + iy, x-iy)$ with itself $\frac{n+2}{2}$ times. The Chern character of this matrix factorization is $(-2i)^{\frac{n+2}{2}}dx_0 \cdots dx_{n+1}$, and so this class gives a basis for $\Hdg(\Dsingdg(R))$.   
\end{ex}

\begin{ex}
\label{ex:cubic}
Now suppose $m = 3$ and $n = 2$, so that $f = x_0^3 + x_1^3 + x_2^3 + x_3^3$. In this case, we have
$$
B = \{(1,1,2,2), (1,2,1,2), (1,2,2,1), (2, 1,1,2), (2, 1, 2, 1), (2, 2, 1,1)\},
$$
and so $\dim_\Q \Hdg(X^2_3) = 6$.  As discussed in the introduction, the Hodge Conjecture is known to hold for all surfaces, and so $\Dsingdg(R)$ satisfies the nc Hodge condition. It follows that there are six classes in $K_0(\Dsingdg(R))_\Q$ whose Chern characters form a basis of $\Hdg(X^2_3)$: let us now describe these classes in terms of matrix factorizations of $x_0^3 + x_1^3 + x_2^3 + x_3^3$. 

Let $\a = e^{2\pi i / 3}$. Taking tensor products of the two matrix factorizations
\begin{align*}
E_1(x_0, x_1) &= \left(x_0 + x_1, (x_0 + \a x_1)(x_0 + \a^2 x_1)\right) \\
E_2(x_0, x_1) &= \left(x_0 + \a x_1, (x_0 + x_1)(x_0 + \a^2 x_1)\right)
\end{align*}
of $x_0^3 + x_1^3$ yields the following six matrix factorizations of $x_0^3 + x_1^3 + x_2^3 + x_3^3$: 
$$
E_1(x_0, x_1) \otimes E_1(x_2, x_3), \quad E_2(x_0, x_1) \otimes E_1(x_2, x_3), \quad E_1(x_0, x_1) \otimes E_2(x_2, x_3),
$$
$$
E_2(x_0, x_1) \otimes E_2(x_2, x_3), \quad E_1(x_0, x_2) \otimes E_1(x_1, x_3), \quad E_1(x_0,x_2) \otimes E_2(x_1, x_3).
$$
Let us compute the Chern characters of these objects. We have:
\begin{align*}
\ch_{HP}(E_1(x_0, x_1)) & = 3(x_1 - x_0) dx_0 dx_1.\\
\ch_{HP}(E_2(x_0, x_1)) & = 3\a(\a x_1 - x_0 ) dx_0 dx_1.
\end{align*}
Thus, letting $\omega \ce dx_0dx_1dx_2dx_3$, the Chern characters of our six matrix factorizations of $x_0^3 + x_1^3 + x_2^3 + x_3^3$ are:
$$
9(x_1 - x_0)(x_3 - x_2)\omega, \quad 9\a(\a x_1 - x_0)(x_3 - x_2)\omega, \quad 9\a(x_1 - x_0)(\a x_3 - x_2)\omega,
$$
$$
9\a^2(\a x_1 - x_0)(\a x_3 - x_2)\omega, \quad 9(x_2 - x_0)(x_3 - x_1)\omega, \quad 9 \a (x_2 - x_0)(\a x_3 - x_1)\omega.
$$
A straightforward calculation shows that these classes are $\Q$-linearly independent and therefore form a basis of $\Hdg(\Dsingdg(R))$. 
\end{ex}

Example~\ref{ex:cubic} shows that every Hodge class of $\Dsingdg(\C[x_0, \dots, x_3] / (x_0^3 + \cdots + x_3^3))$ can be built out of products of Hodge classes of $\Dsingdg(\C[x_0, x_1] / (x_0^3 + x_1^3))$. The next example shows that this isn't always the case for Fermat hypersurfaces, even in four variables.

\begin{ex}
We now take $m = 6$ and $n = 2$. We have $(2, 2, 3, 5) \in B$ in this case. Notice that $(2, 2, 3, 5)$ is not the concatenation of elements of $(\Z/6)^2$ corresponding to Hodge classes of $X^0_6$. This implies that the Hodge class corresponding to $(2, 2, 3, 5)$ cannot arise as the product of Hodge classes of $\Dsingdg(\C[x_0, x_1] / (x_0^3 + x_1^3))$. Indeed, we do not know how to explicitly express this Hodge class as a $\C$-linear combination of Chern characters of matrix factorizations, even though, since the Hodge conjecture holds for $X^2_6$, this must be possible.
\end{ex}

\bibliographystyle{amsalpha}
\bibliography{Bibliography}
\end{document}